\documentclass{siamonline220329}

\usepackage{lipsum}
\usepackage{amsfonts}
\usepackage{graphicx}
\usepackage{epstopdf}
\usepackage{algorithmic}
\usepackage{mathrsfs}
\usepackage{amsmath}

\ifpdf
  \DeclareGraphicsExtensions{.eps,.pdf,.png,.jpg}
\else
  \DeclareGraphicsExtensions{.eps}
\fi

\usepackage{enumitem}
\setlist[enumerate]{leftmargin=.5in}
\setlist[itemize]{leftmargin=.5in}

\newsiamremark{remark}{Remark}
\newsiamremark{hypothesis}{Hypothesis}
\crefname{hypothesis}{Hypothesis}{Hypotheses}
\newsiamthm{claim}{Claim}

\headers{Optimal Control With Carbon Allowances Purchasing}{Xinfu Chen, Yuchao Dong, Wenlin Huang, and Jin Liang}

\title{Optimal Carbon Emission Control With Allowances Purchasing\thanks{Submitted to the editors July 6, 2024.
\funding{The work of the first, third and fourth authors was supported by National Natural Science
Foundation of China grant 12071349. The work of the second author was supported by National
Natural Science Foundation of China grants 12071333 and 12101458.}}}

\author{Xinfu Chen\thanks{School of Mathematical Sciences, Key Laboratory of Intelligent Computing and Applications (Ministry of Education), Tongji University, Shanghai 200092, China, and Department of Mathematics, University of Pittsburgh, Pittsburgh, PA 15260 USA, and School of Mathematics, Southwestern University Of Finance And Economics, Chengdu 611130, China 
  (\email{xinfu@pitt.edu})}
\and Yuchao Dong\thanks{School of Mathematical Sciences, Key Laboratory of Intelligent Computing and Applications (Ministry of Education), Tongji University, Shanghai 200092, China 
  (\email{ycdong@tongji.edu.cn}).}
\and Wenlin Huang\thanks{Corresponding author. College of Science, University of Shanghai for Science and Technology,  Shanghai 200093, China 
  (\email{huangwenlin@usst.edu.cn}).}
  \and Jin Liang\thanks{School of Mathematical Sciences, Key Laboratory of Intelligent Computing and Applications (Ministry of Education), Tongji University, Shanghai 200092, China  
  (\email{liang\_jin@tongji.edu.cn}).}}

\usepackage{amsopn}

\begin{document}

\maketitle

\begin{abstract}
In this paper, we consider a company can simultaneously reduce its emissions and buy carbon allowances at any time. We establish an optimal control model involving two stochastic processes with two control variables, which is a singular control problem. This model can then be converted into a Hamilton-Jacobi-Bellman (HJB) equation, which is a two-dimensional variational equality with gradient barrier, so that the free boundary is a surface. We prove the existence and uniqueness of the solution. Finally, some numerical results are shown.
\end{abstract}



\section{Introduction}

The escalating impact of the greenhouse effect on global climate change has led to increasingly severe consequences. In 2023, the world witnessed its hottest summer on record since data collection began in 1940, as reported by the European Union’s climate monitoring agency. The frequency of extreme weather events and natural disasters, such as heatwaves, droughts, floods, and wildfires, has surged. Recognizing the urgent need to address climate change's repercussions on the global economy and society, international agreements, including the United Nations Framework Convention on Climate Change, the Kyoto Protocol, and the Paris Agreement, have been adopted since 1992. To date, major economies worldwide, including United States, China and European Union,  have each put forward their own carbon emission reduction plans.

In summary, the global climate governance framework has progressively solidified. In the face of uncertainties on carbon emissions and allowance prices, efficiently utilizing limited resources becomes a crucial mathematical challenge in carbon reduction research. Currently, many scholars are delving into comprehensive studies on optimizing strategies for carbon emission reduction. For example, Yang et al.\cite{A1} develop a multi-stage mixed integer nonlinear programming model to minimize the trading cost and implementation cost of green technology under cap-and-trade schemes. Bourgey et al.\cite{A2} combine climate risk with credit risk and establish an optimization model of an enterprise's emission level under the dual goals of maximizing production profits and complying with emission reduction scenarios. Wang et al.\cite{A3} establish a framework to determine the scale of investment required for each pathway by analyzing the emission reduction pathways of enterprises. Rubin\cite{A4}, Schennach\cite{A5} and Liski\cite{A6} study the optimal emission reduction policies for firms that could store emission allowances over multiple trading periods. Yang and Liang \cite{A7} study the optimal control of carbon emission reduction from the perspective of undefined equity for the first time. They add control factors into the undefined equity pricing model under the consumption framework of Merton\cite{A8,A9} and establish the optimal control model of carbon emission reduction. Later on, Yang et al.\cite{A10} establish two models based on two emission reduction strategies and derive the corresponding HJB equation according to the dynamic programming principle \cite{A11,A12,A13}. They  prove the existence and uniqueness of  classical solutions with the theory of parabolic PDEs. On this basis, Guo and Liang \cite{A14} consider carbon trading problem and assume that the allowance price on the market satisfies the geometric Brownian motion. Then, they establishe a stochastic control model to minimize the total cost of emission reduction, and obtain the semi-closed solution of the corresponding HJB equation. In addition, Guo and Liang \cite{A15} also considere the penalty for excess emissions and the upper limit of emission reduction capacity, and prove the existence and uniqueness of the viscosity solution of the HJB equation corresponding to the problem. Then, based on Guo's work, Liang and Huang \cite{A16} study the impact of intertemporal trading on the company's emission reduction cost and emission reduction strategy. In addition, in order to study the influence of Poisson jump on the investment timing of carbon emission reduction technologies, based on the real option theory, Huang et al.\cite{A17} establish the optimal investment timing model for carbon emission reduction technology when the investment cost satisfy a compound Poisson process, and the semi-closed solution of the optimal implementation boundary was obtained by using the iteration technique of the basic function. In addition, Liang and Huang \cite{A18} studied the stochastic optimization control problem of carbon emission reduction, and conducted theoretical research on the corresponding HJB equation through viscosity solution.

Recently, Huang et al.\cite{A19} study a stochastic optimization control problem for carbon emission reduction with two control variables, one of which is the emission reduction control strategy, and the other is the purchase amount of carbon allowances. The existence and uniqueness of classical solution of the model and the optimal value of the model are obtained. However, in previous work\cite{A19}, the purchase and auction of carbon allowances are assumed to be carried out at a specific time point, while the actual trading time is usually not fixed, and companies may participate in the trading of allowances at multiple time points. In this way, while reducing emissions, enterprises should also take into account the price of allowances in the current market, and buy or sell allowances at any time, so that the total cost of emission reduction is lowest. Therefore, in this paper, we establish an optimal control model in which the company can simultaneously reduce its emissions and buy carbon allowances at any time, so that the model has two control variables, one of which is the stochastic emission reduction control strategy, and the other is the purchase amount of carbon allowances. We then derive the corresponding HJB equation, which is a parabolic variational inequality on two dimensional state space with gradient constraint, so that the free boundary is a two-dimensional surface. To show that the model makes rational sense, we need to prove the existence and uniqueness of the solution. Finally, some properties of the free boundary surface are discussed and some numerical results are shown.

The remainder of the paper is organized as follows: In Section \ref{sect_prelim}, we introduce the basic setup for our optimal control problem and derive the related HJB equation, which is a variational inequality. By considering the penalized equation, we show the existence of solution in Section \ref{sect_exist}. In Section \ref{sect_uniq}, we prove the comparison principle for the related equation, which implies the uniqueness of solution. Some numerical results are demonstrated in Section \ref{sect_numer}.

\section{Preliminaries}\label{sect_prelim}
Within the production process, an enterprise generates carbon dioxide and is obliged to manage the quantity of carbon emissions.  It is assigned a certain amount of  emission allowance. In the end of a stage, should the total emissions surpass the allocated quota, the company must acquire additional quotas from the market. Through the time period, the enterprise can opt to reduce emissions through internal measures or increase its emissions quota by purchasing carbon allowances in the carbon trading market. Moreover, based on the current market conditions, we also assume that the company is not eligible to sell carbon quotas. The company aims to minimize the total cost of carbon emissions by striking a balance between implementing internal emission reduction measures and purchasing carbon allowances in the carbon market.
\subsection{Problem Formulation}
The mathematical modelling for the problem is demonstrated as follows. Let $(\Omega,\mathscr{F},\mathbb{P})$ be a probability space on which two independent standard Brownian motions $W$ and $Z$ are defined. The $\mathbb{F}=\{\mathscr{F}_t,t>0\}$ is the natural  filtration generated by $W$ and $Z$, which satisfies the usual conditions.  Let $A^1$ denote the set of all $\mathbb{F}$-adapted non-negative processes, representing the rate of reduction of carbon emissions through the enterprise's internal measures.  Similarly, denote by $A^2$ the totality of all $\mathbb{F}$-adapted right continuous increasing processes,  representing the cumulative purchased quota. 

We assume that the unit price $S$ of carbon allowance quota follows Geometric Brownian Motion, i.e. 
\begin{align}\label{1}
dS_\theta=S_\theta(\mu d\theta+\sigma dW_\theta).
\end{align}
Given $a\in A^1$ and $b \in A^2$, the carbon emission surplus $X$ satisfies the following stochastic equation  
\begin{align}\label{2}
dX_\theta=-a_\theta d\theta-db_\theta+\nu dZ_\theta.
\end{align}
Starting from time $t$ with initial condition $X_t=x$ and $S_t=s$, the total cost for the enterprise can be represented by
\begin{equation}
C(x,s,t;a,b)=X_T^+S_Te^{-r(T-t)}+\int_t^Te^{-r(\theta-t)}S_\theta db_\theta+\int_t^Te^{-r(\theta-t)}c(a_\theta)d\theta,
\end{equation}
where $r>0$ indicates the discount rate. The equation above encompasses three cash flows on the right side. The first part represents the cost of purchasing quota if the emissions exceed the quota at time $T$. The second part accounts for the cost of purchasing quota during the interval 
$[t,T]$, and the third part signifies the cost of the enterprise's internal measures. In this context, we assume that the internal control cost is a function of the reduction rate and is defined as $c(a)=\frac{m}{2}a^2 $, where $m$ is a positive constant. Such an assumption has also been utilized in \cite{A10}. The goal of the company is to  minimize the expected discounted  cost
\begin{align}\label{4}
p(x,s,t)=\inf_{a\in A^1,b\in A^2}\mathbb E[C(x,s,t;a,b)|X_t=x,S_t=s],
\end{align}
for any $(x,s)\in\mathbb{R}\times\mathbb{R}^+,0\leq t\leq T$.

{\bf Notation} In this paper, we
use $C$ to represent a constant, but may be different from line to line. In general, it depends on the coefficients of the model. In the proof of theorems, we use the notation $C(\cdot)$ to indicate its dependence on other quantities. For example, $C(\varepsilon)$ means that the constant depends also on $\varepsilon$.
\subsection{Basic Property of the Value Function}
Before solving the value function $p(x,s,t)$, let us first study the basic properties of the value function.
\begin{lemma}
In the distributional sense, we have that
\[0\leq \frac{\partial p}{\partial x}\leq se^{(\mu-r)t}, \quad 0\leq\frac{\partial p}{\partial s}\leq\frac{p}{s}.\]
\end{lemma}
\begin{proof}
Fix $a\in A^1$ and $b\in A^2$. Conditioning on $X_t=x$, we have
\[X_\theta=x+\nu(Z_\theta-Z_t)-\int_t^\theta a_\tau d\tau-(b_\theta-b_t).\]
Hence, for any $h>0$, denoting by $C(x,s,t;a,b)$ and $C(x+h,s,t;a,b)$ the cost conditioned on $(X_t,S_t)=(x,s)$ and $(X_t,S_t)=(x+h,s)$ respectively, we see that
\[0\leq C(x+h,s,t;a,b)-C(x,s,t;a,b)\leq hS_Te^{-r(T-t)}.\]
Taking expectation, it holds that
\[0\leq \mathbb E[C(x+h,s,t;a,b)]-\mathbb E[C(x,s,t;a,b)]\leq he^{-r(T-t)}\mathbb E[S_T]=hse^{(\mu-r)(T-t)}.\]
Taking the infimum  with respect to $(a,b)$, it is easy to obtain that 
\[0\leq \frac{p(x+h,s,t)-p(x,s,t)}{h}\leq se^{(\mu-r)(T-t)} \quad \forall x\in\mathbb{R},s>0,t\leq T.\]
Letting $h\rightarrow 0$, we have
\[0\leq \frac{\partial p}{\partial x}\leq se^{(\mu-r)t}.\]

Since $S_\theta=S_t e^{(\mu-\frac{\sigma^2}{2})(\theta-t)+\sigma(W_\theta-W_t)}$, we get  that, for any $h>0$,
\[C(x,s,t;a,b)\leq C(x,(1+h)s,t;a,b)\leq(1+h)C(x,s,t;a,b).\]
It follows that
\[p(x,s,t)\leq p(x,(1+h)s,t)\leq(1+h)p(x,s,t)\quad \forall x\in\mathbb{R},s>0,t\leq T,h>0,\]
which is equivalent to 
\[0\leq \frac{p(x,(1+h)s,t)-p(x,s,t)}{sh}\leq \frac{hp(x,s,t)}{sh} \quad \forall x\in\mathbb{R},s>0,t\leq T.\]
Letting $h\rightarrow 0$, we have
\[0\leq\frac{\partial p}{\partial s}\leq\frac{p}{s}.\]
\end{proof}

\begin{lemma}
\[0\leq p(x,s,t)\leq se^{(\mu-r)(T-t)}\Big[x^++\frac{v\sqrt{T-t}}{\sqrt{2\pi}}\Big].\]
\end{lemma}
\begin{proof}
The first inequality is trivial, as $C(x,s,t;a,b)$ is always non-negative. 
Taking the trivial strategy $a\equiv0$ and $b\equiv0$, we see that
\begin{align}
p(x,s,t)&\leq \mathbb  E[(x+\nu(Z_T-Z_t))^+se^{(\mu-\frac{\sigma^2}{2})(T-t)+\sigma(W_T-W_t)-r(T-t)}]\notag\\
&=se^{(\mu-r)(T-t)}\int_{-\infty}^\infty(x+\nu\sqrt{T-t}\xi)^+\frac{1}{\sqrt{2\pi}}e^{-\frac{\xi^2}{2}}d\xi\notag\\
&=se^{(\mu-r)(T-t)}\Big[\int_{-\frac{x}{\nu\sqrt{T-t}}}^\infty x\frac{1}{\sqrt{2\pi}}e^{-\frac{\xi^2}{2}}d\xi+\nu\sqrt{T-t}\int_{-\frac{x}{\nu\sqrt{T-t}}}^\infty \xi\frac{1}{\sqrt{2\pi}}e^{-\frac{\xi^2}{2}}d\xi\Big]\notag\\
&\leq se^{(\mu-r)(T-t)}\Big[x^++\frac{\nu\sqrt{T-t}}{\sqrt{2\pi}}e^{-\frac{x^2}{2\nu^2(T-t)}}\Big].\notag
\end{align}
The lemma is proved.
\end{proof}

\begin{lemma}
For fixed $s>0$, $t<T$, $p(\cdot,s,t)$ is a convex function. 
\end{lemma}
\begin{proof}
Choosing any $x_1\in\mathbb{R}$, $x_2\in\mathbb{R}$ and $\xi\in[0,1]$, set $x_3=\xi x_1+(1-\xi)x_2$. Let $(a_1,b_1)$ and $(a_2,b_2)$ be two strategies in $A^1\times A^2$ and construct a new strategy $(a_3,b_3)$ as 
\[a_3=\xi a_1+(1-\xi)a_2,\quad b_3=\xi b_1+(1-\xi)b_2.\]
Denote by $\{X_\theta^i\}_{t\leq \theta\leq T},i=1,2,3,$ be the solution of \eqref{2} under the strategy $(a_i,b_i)$ with initial condition $S_t=s$ and $X_t^i=x_i$. Then we see that
\[X_t^3=\xi X_t^1+(1-\xi)X_t^2.\]
Since the functions $f(x):=x^+$ and $g(a):=a^2$ are convex, we see that
\[C(x_3,s,t;a_3,b_3)\leq \xi C(x_1,s,t;a_1,b_1)+(1-\xi)C(x_2,s,t;a_2,b_2).\]
This implies that
\[p(x_3,s,t)\leq \mathbb E[C(x_3,s,t;a_3,b_3)]\leq \xi \mathbb E[C(x_1,s,t;a_1,b_1)]+(1-\xi)\mathbb E[C(x_2,s,t;a_2,b_2)].\]
Taking the infimum over $(a_1,b_1)$ and $(a_2,b_2)$, we then obtain
\[p(\xi x_1+(1-\xi)x_2,s,t)\leq\xi p(x_1,s,t)+(1-\xi)p(x_2,s,t).\]
This means that $p(\cdot,s,t)$ is convex in the first variable.
\end{proof}
\subsection{HJB Variational Inequality}
From standard optimal stochastic control theory, one can derive that 
the value function $p$ is related to the following variational inequality:
\begin{align}\label{888}
\begin{cases}
\max\{\mathcal M[p](x,s,\theta),\frac{\partial p}{\partial x}(x,s,\theta)-s\}=0 \quad \forall x\in\mathbb{R},s>0, \theta\leq T,\\
p(x,s,T)=x^+s \quad \forall x\in\mathbb{R},s>0,
\end{cases}
\end{align}
where
\[\mathcal M [p](x,s,\theta)=-\frac{\partial p}{\partial \theta}-\frac{\sigma^2}{2}s^2\frac{\partial^2p}{\partial s^2}-\mu s\frac{\partial p}{\partial s}+rp-\frac{\nu^2}{2}\frac{\partial^2p}{\partial x^2}+\frac{1}{2m}\Big(\frac{\partial p}{\partial x}\Big)^2.\]
In fact, we are going to prove the solution of (\ref{888}) is just the $p$ defined in (\ref{4}), i.e., we shall have the following verification theorem.
\begin{theorem}
Let $\Phi$ be a $C^{2,1}$ solution of \eqref{888} and satisfying the growth condition, for some constant $C$ and  $n\in\mathbb{N}$,
\begin{equation}\label{growth_condi}
    \Big|\frac{\partial \Phi}{\partial x}\Big|+\Big|\frac{\partial \Phi}{\partial s}\Big|\leq C(s^n+s^{-n})\quad \forall x\in\mathbb{R},s>0,t\in[0,T].
\end{equation}
Then, $\Phi(x,s,t) \le p(x,s,t)$. Denote by $Q_1:=\{(x,s,t)| \frac{\partial \Phi}{\partial x}(x,s,t) < s\}$. For any $(x,s,t) \in \bar Q_1$, assume that there exists $a^* \in A^1$ and $b^* \in A^2$ with continuous paths such that 
\begin{enumerate}
    \item $a^*_\theta= \frac{1}{m} \frac{\partial \Phi}{\partial x}(X^*_\theta,S_\theta,\theta)$ a.s.;
    \item $(X^*_\theta,S_{\theta},\theta) \in \bar Q_1$ for all $\theta \in [t,T]$ almost surely;
    \item $\int_t^T {\bf 1}_{\{(X^*_\theta,S_\theta,\theta) \in Q_1\}} db^*_{\theta}=0$, a.s.,
\end{enumerate}
where $X^*$ is the solution of \eqref{2} with control $(a^*,b^*)$. Then, $\Phi(x,s,t)=\mathbb E\left[ C(x,s,t;a^*,b^*)\right]$, which implies that $(a^*,b^*)$ is the optimal control. For $(x,s,t) \notin \bar Q_1$, the optimal control should immediately buy $\Delta$ units emission quota such that $(x-\Delta,s,t) \in \bar Q_1$ and apply $(a^*,b^*)$. 
\end{theorem}
\begin{proof}
Given $(a,b)\in A^1\times A^2$, $\{(X_\theta,S_\theta)\}_{\theta\geq 0}$ is the solution of \eqref{1} and \eqref{2} subject to $X_t=x$, $S_t=s$. By It\^{o}'s formula, 
\begin{align}\label{ito1}
&e^{-r(T-t)}\Phi(X_T,S_T,T)-\Phi(x,s,t)\notag\\
=&\int_t^Te^{-r(\theta-t)}\Big[\Big(\frac{\partial\Phi}{\partial \theta}+\mu S_\theta\frac{\partial\Phi}{\partial s}+\frac{1}{2}\sigma^2S_\theta^2\frac{\partial^2\Phi}{\partial s^2}-a_\theta\frac{\partial\Phi}{\partial x}+\frac{\nu^2}{2}\frac{\partial^2\Phi}{\partial x^2}-r\Phi\Big)d\theta\notag\\
&-\int_t^T e^{-r(\theta-t)}\frac{\partial\Phi}{\partial x}db^c_\theta +\int_t^Te^{-r(\theta-t)}\Big[\sigma S_\theta\frac{\partial\Phi}{\partial s}dW_\theta+\nu\frac{\partial\Phi}{\partial x}dZ_\theta\Big]\notag\\
&+\sum_{\theta}e^{-r(\theta-t)}\left\{\Phi(X_{\theta-}-\Delta_\theta b,S_\theta,\theta)- \Phi(X_{\theta-},S_\theta,\theta)\right\},
\end{align}
where $b^c$ represents the continuous part of $b$ and $\Delta_{\theta}b:=b_{\theta}-b_{\theta-}$ is the jump size of $b$ at time $\theta$. From the growth condition \eqref{growth_condi} of $\Phi$, we have that 
\begin{align}
\mathbb E\Big[\int_t^T\Big(s^2\Big|\frac{\partial \Phi}{\partial s}\Big|^2+\Big|\frac{\partial \Phi}{\partial x}\Big|^2\Big)dt\Big]<\infty.\label{6}
\end{align}
Then, 
\begin{align}\label{ito2}
\Phi(x,s,t)=&\mathbb E\big[e^{-r(T-t)}\Phi(X_T,S_T,T)+\int_t^Te^{-r(\theta-t)}\frac{m}{2}a_\theta^2dt+\int_t^Te^{-r(\theta-t)}S_\theta db_\theta\notag\\
&+\int_t^Te^{-r(\theta-t)}\mathcal M[\Phi](X_\theta,S_\theta,\theta) d\theta-\int_t^Te^{-r(\theta-t)}\frac{m}{2}\Big(a_\theta-\frac{1}{m}\frac{\partial\Phi}{\partial x}\Big)^2d\theta\notag\\
&-\int_t^Te^{-r(\theta-t)}\Big(S_\theta-\frac{\partial\Phi}{\partial x}\Big)db^c_\theta\notag\\
&-\sum_{\theta} e^{-r(\theta-t)}\left\{S_\theta \Delta_\theta b+\Phi(X_{\theta-}-\Delta_\theta b,S_\theta,\theta)- \Phi(X_{\theta-},S_\theta,\theta)\right\}\big].
\end{align}
Since $\frac{\partial\Phi}{\partial x} \le s$, we have that 
$$
S_\theta \Delta_\theta b+\Phi(X_{\theta-}-\Delta_\theta b,S_\theta,\theta)- \Phi(X_{\theta-},S_\theta,\theta) \ge 0.
$$
Recalling that $\Phi$ is the solution of \eqref{888}, we get 
$$
\Phi(x,s,t) \le \mathbb E\left[C(x,s,t;a,b)\right]
$$
for any $a\in A^1$ and $b \in A^2$, which implies that $\Phi(x,s,t)\le p(x,s,t)$. Finally, if one can find $(a^*,b^*)$ satisfying the condition in the theorem, then the last four terms in \eqref{ito2} will be identical to $0$. Hence, $\Phi(x,s,t)=\mathbb E\left[C(x,s,t;a^*,b^*)\right]\geq p(x,s,t)$.  
\end{proof}

\section{Existence of the Solution}\label{sect_exist}
In this section, we will prove the existence of a solution for (\ref{888}). First, let us make some transformation. Define $u(x,y,t)=\frac{\Phi(x,e^y,T-t)}{e^y}$. We see that $u$ is a solution of
\begin{align}\label{9}
\begin{cases}
\max\{\frac{\partial u}{\partial x}-1,\frac{\partial u}{\partial t}-\frac{\sigma^2}{2}\frac{\partial^2u}{\partial y^2}-(\mu+\frac12\sigma^2)\frac{\partial u}{\partial y}+(r-\mu)u-\frac{\nu^2}{2}\frac{\partial^2u}{\partial x^2}+\frac{e^y}{2m}(\frac{\partial u}{\partial x})^2\}=0,\\
u(x,y,0)=x^+.
\end{cases}
\end{align}
For notational simplicity, we introduce the following operators:
\[\mathcal{L}[\psi]=\frac{\partial}{\partial t}\psi-\frac{\sigma^2}{2}\frac{\partial^2}{\partial y^2}\psi-\frac{\nu^2}{2}\frac{\partial^2}{\partial x^2}\psi-(\mu+\frac12\sigma^2)\frac{\partial}{\partial y}\psi-(\mu-r)\psi,\]
\[\mathcal{A}[\psi]=\mathcal{L}\psi+\frac{e^y}{2m}\Big(\frac{\partial\psi}{\partial x}\Big)^2,\]
and
\[\mathcal{B}[\psi]=\mathcal{L}\psi+\frac{e^y}{m}\psi\frac{\partial\psi}{\partial x},\]
and denote
\[\Omega=\mathbb R^2,\quad Q=\Omega\times(0,T].\]
Noting that $u$ satisfies a variational inequality with gradient constraint. To obtain its existence, people usually consider the equation satisfied by its derivative. Thus, setting $v=\frac{\partial u}{\partial x}$, it should solve the following obstacle problem
\begin{align}
\begin{cases}
\max\{v-1,\mathcal{B}[v]\}=0\quad \mbox{in} \quad Q,\\
v(x,y,0)=\mathbf{1}_{\{x>0\}}.\label{10}
\end{cases}
\end{align}
However, due to the discontinuity of the initial value, $v$ will not be continuous at the point $x=0,t=0$. To characterize its behaviour at the initial time,  We give the following modified definition of the solution.
\begin{definition}
We say a function $v\in L^\infty(Q)\cap(\cap_{p>1}W_{p,loc}^{2,1}(Q))$ is a solution of 
\eqref{10}, if
\begin{enumerate}
    \item $\max\{v-1,\mathcal{B}[v]\}=0$, a.e. in $\Omega\times(0,T]$
    \item For any $K>0$,
\begin{equation}\label{intial_cond}
  \lim_{t\rightarrow 0} \sup_{|y|\le K,x\in \mathbb R}|v(x,y,t)-\varphi(\frac{x}{\nu \sqrt{t}})|=0,
 \end{equation}
where $\varphi(x)=\int_{-\infty}^x \frac{1}{\sqrt{2\pi}}e^{-\frac{u^2}{2}}du$ is the cumulative distribution function of standard Gaussian distribution.
\end{enumerate}
\end{definition}
\begin{remark}
\begin{enumerate}
\item Dai et al \cite{dai2022nonconcave} have defined same initial condition for viscosity solutions of HJB equations with discountinuous initial condition, including $\bf 1_{\{x>0\}}$.
    \item From \eqref{intial_cond}, we see that $v$ is continuous at the points $(x,y,0)$ with $x \neq 0$. It is also continuous at time $t=0$ in the weak sense, i,e. for any $\phi \in C_c^\infty(\Omega)$,
    $$
    \lim_{t\rightarrow0} \int_{\Omega}v(x,y,t)\phi(x,y)dxdy=\int_{\Omega}\mathbf{1}_{\{x>0\}} \phi(x,y) dxdy.
    $$
    If we further know the uniform  asymptotic behaviour of $v$ at $x=-\infty$, one can also have that
    $$
    \lim_{t\rightarrow 0}\int_{-\infty}^x v(z,y,t)dz=
    \int_{-\infty}^x{\bf 1}_{\{z>0\}}dz=x^+.
    $$
\end{enumerate}
\end{remark}

For the coefficients of the equations, we assume that $\sigma,\nu,\mu,r$ and $m$ are all positive constants. Moreover, we assume that $$\mu>r.$$
Otherwise, let $v$ be solution of $\mathcal{B}[v]=0$ subject to $v(x,y,0)=\mathbf{1}_{\{x>0\}}$. With the condition $\mu<r$, by comparison principle, we find that $v\leq 1$, which implies that it is also the solution of the variational inequality. This means that it is never optimal for the enterprise to purchase quotas before the terminal time as the discount rate is greater than the mean growth rate of the price for carbon allowance.    


\subsection{The Penalized Equation}
Following the idea of well-developed theory of obstacle problem (see \cite{A20} for instance), we consider a penalized
approximation of problem (\ref{10}),
\begin{align}\label{11}
\begin{cases}
\mathcal B[v^\varepsilon]+\beta(\frac{1+\varepsilon-v^\varepsilon}{\varepsilon})=0\quad \mbox{in} \quad Q,\\
v^\varepsilon(x,y,0)=v_0^\varepsilon(x)\quad \forall(x,s)\in\Omega.
\end{cases}
\end{align}
where $\varepsilon\in(0,1]$, $\{v_0^\varepsilon(\cdot)\}_{\varepsilon\in(0,1]}$ satisfies
$v_0^\varepsilon\in C^\infty(\mathbb{R})$, $0\leq v_0^\varepsilon\leq 1+\varepsilon$, $\frac{\partial v_0^\varepsilon}{\partial x}>0$ when $0<v_0^\varepsilon<1+\varepsilon$,
$\frac{\partial v_0^\varepsilon}{\partial \varepsilon}\geq 0$, $v_0^\varepsilon\leq e^{\varepsilon+x}$, $v_0^\varepsilon(x)=1+\varepsilon$ for $x\geq 0$, $v_0^\varepsilon(x)\equiv 0$ for $x\leq -\varepsilon$. The function $\beta(\cdot)$ satisfies the following conditions
$$\beta\in C^\infty(\mathbb{R}),\quad \beta''=0 \quad  \mbox{on}\quad  (-\infty,0], \quad \beta=0 \quad \mbox{on} \quad [1,\infty),$$
$$\beta'\leq 0 \quad \mbox{on} \quad \mathbb{R}, \quad\beta''\geq 0\quad  \mbox{on}\quad  \mathbb{R},\quad\beta(0)=2(\mu-r).$$

Then, we have the following results for the existence of the solution of \eqref{11}.
\begin{lemma}\label{lem_penalty}
There exists a classical smooth solution $v^\varepsilon(x,y,t)$ of \eqref{11}. Moreover, it holds that, for $(x,y,t) \in Q$, 
$$
\frac{\partial v^\varepsilon}{\partial x} >0,
$$
\end{lemma}
\begin{proof}
To prove the monotonicity with respect to $x$, we will use comparison principle for parabolic non-linear equation. However, due to the non-linear term $\frac{e^y}{m}v\frac{\partial v}{\partial x}$ in $\mathcal B[v]$, it does not satisfy the standard assumption for comparison principle (see \cite[Theorem 9.1]{lieberman1996second} for instance). For that purpose, let us consider a modified operator $\tilde {\mathcal B}$ as 
$$
\tilde {\mathcal B }[v]= \mathcal L[v]+\frac{e^y}{m}v(\frac{\partial v}{\partial x})_+,
$$
where $(x)_+$ stands for the positive part of $x$. By Schauder's fixed point Theorem \cite{A21,A22}, there exists a classical smooth solution $v^\varepsilon$ for
\begin{align*}
\begin{cases}
\tilde {\mathcal B}[v^\varepsilon]+\beta(\frac{1+\varepsilon-v^\varepsilon}{\varepsilon})=0\quad \mbox{in} \quad Q,\\
v^\varepsilon(x,y,0)=v_0^\varepsilon(x)\quad \forall(x,s)\in\Omega.
\end{cases}
\end{align*}
For any $x'>0$, denote $\tilde v^{\varepsilon}(x,y,t):=v^{\varepsilon}(x+x',y,t)$ and it holds that $\tilde v^\varepsilon$ also satisfies 
$$
\tilde {\mathcal B}[\tilde v^\varepsilon]+\beta(\frac{1+\varepsilon-\tilde v^{\varepsilon}}{\varepsilon})=0
$$
with initial condition $\tilde v^\varepsilon(x,y,0)=v_0^\varepsilon(x+x')$. We see that the nonlinear operator $\tilde {\mathcal B}[v]+\beta(\frac{1+\varepsilon-v}{\varepsilon})$ satisfies the assumptions for comparison principle. Since $v_0^\varepsilon$ is non-decreasing with respect to $x$, one deduce that $\tilde v^\varepsilon \ge v^\varepsilon$, which is equivalent to $\frac{\partial v^\varepsilon}{\partial x} \ge 0$. 
Moreover,  it implies that $v^{\varepsilon}$ satisfies 
$$
\mathcal B[v^\varepsilon]+\beta(\frac{1+\varepsilon-v^\varepsilon}{\varepsilon})=0.
$$
That is, there exists a classical smooth solution $v^\varepsilon(x,y,t)$ of \eqref{11}.

Letting $w=\frac{\partial v^\varepsilon}{\partial x}$, we have
\begin{align}\label{5555}
\begin{cases}
\mathcal L[w]+\frac{e^y}{m}v^\varepsilon w_x+(\frac{e^y}{m}w-\beta'(\frac{1+\varepsilon-v^\varepsilon}{\varepsilon})\frac{1}{\varepsilon})w=0 \quad\mbox{in} \quad Q,\\
w(x,y,0)=\frac{\partial v_0^\varepsilon}{\partial x}\quad \forall(x,y)\in\Omega.
\end{cases}
\end{align}
Since we have $w\ge 0$ and $\beta' \le 0$, we see that the coefficient for $w$ term is no less than $0$. Then, one can apply 
strong maximum principle shows that $w> 0$.
\end{proof}
\begin{remark}\label{rmk_compare}
Assume that $u$ and $v$ satisfy 
    $$
    \mathcal B[u]+\beta(\frac{1+\varepsilon-u}{\varepsilon}) \ge 0,
    $$
    and
    $$
    \mathcal B[v]+\beta(\frac{1+\varepsilon-v}{\varepsilon}) \le 0,
    $$
respectively, and $v \le u$ at time $t=0$. If we further know that $u$ and $v$ are both non-decreasing with respect to $x$, then, from previous proof, it holds that $v \le u$. In fact, we also give a comparison principle for related variational inequality in Section \ref{sect_uniq}. 
\end{remark}
Now, we give some prior estimates of the approximating solution $v^\varepsilon(x,s,t)$, which are independent of $\varepsilon$.


\begin{lemma}\label{lemma5.1}
$0\leq v^\varepsilon(x,y,t)\leq 1+\varepsilon$, $\forall \varepsilon\in(0,1]$, $t\geq0$, $x,y\in\mathbb{R}$.

 $\quad\quad\quad\quad\quad 0\leq \beta(\frac{1+\varepsilon-v^\varepsilon}{\varepsilon})\leq2(\mu-r)$ $\forall \varepsilon\in(0,1]$, $t\geq0$, $x,y\in\mathbb{R}$.
\end{lemma}
\begin{proof}
Set $\underline{v}\equiv0$ and $\overline{v}\equiv1+\varepsilon$, we have
\[\mathcal B[\underline{v}]+\beta(\frac{1+\varepsilon-\underline v}{\varepsilon})=\beta(\frac{1+\varepsilon}{\varepsilon})\equiv0,\]
\[\mathcal B[\overline{v}]+\beta(\frac{1+\varepsilon-\overline v}{\varepsilon})=-(\mu-r)(1+\varepsilon)+\beta(0)\ge 0,\]
\[\underline{v}(x,y,0)=0\leq v^\varepsilon(x,y,0).\]
\[\overline{v}(x,y,0)=1+\varepsilon\geq v^\varepsilon(x,y,0).\]
Hence, by maximum principle (see Remark \ref{rmk_compare}), we have
\[0\leq v^\varepsilon(x,y,t)\leq 1+\varepsilon.\]
Then, by the conditions of $\beta$,
$$\quad\quad\quad\quad\quad\quad0\leq \beta(\frac{1+\varepsilon-v^\varepsilon}{\varepsilon})\leq2(\mu-r).$$
\end{proof}
\begin{lemma}\label{lem_mono}
$\frac{\partial v^\varepsilon}{\partial \varepsilon}\geq 0$ $\forall \varepsilon\in(0,1]$, $t\geq0$, $x,y\in\mathbb{R}$.
\end{lemma}
\begin{proof}
Setting $\zeta=\frac{\partial v^\varepsilon}{\partial \varepsilon}$, we obtain $\zeta(\cdot,\cdot,0)=\frac{\partial v^\varepsilon_0}{\partial \varepsilon}\geq 0$.
When $0\leq v^\varepsilon \leq1$, we have $\beta'(\frac{1+\varepsilon-v^\varepsilon}{\varepsilon})=0$, while $1<v^\varepsilon \leq1+\varepsilon$, we have $\beta'(\frac{1+\varepsilon-v^\varepsilon}{\varepsilon})\leq0$. Thus,
\[\mathcal L[\zeta]+\frac{e^y}{m}(v^\varepsilon\zeta_x+\zeta v^\varepsilon_x)-\beta'(\frac{1+\varepsilon-v^\varepsilon}{\varepsilon})\frac{\zeta}{\varepsilon}=\beta'(\frac{1+\varepsilon-v^\varepsilon}{\varepsilon})\frac{1-v^\varepsilon}{\varepsilon^2}\geq 0.\]
Hence, by the maximum principle, $\zeta\geq 0$.
\end{proof}


\begin{lemma}\label{lemma5.6}
For $\delta=\frac{\mu-r}{2m}$, $a>0$ and some $\kappa$ large enough, we have 
\[\rho(\frac{\delta xe^{-\kappa t}}{a+e^y})\leq v^\varepsilon(x,y,t)\leq e^{x+(\frac{\nu^2}{2}+\mu-r)t+\varepsilon}.\]
where
\begin{align}
\rho(z)=\begin{cases}
\frac12(z-1)-\frac{1}{2\pi} \sin(\pi(z-1)),\quad \rm if\quad z\in[1,3],\\
1,\quad \rm if\quad z\geq 3,\\
0,\quad \rm if \quad z\leq1.
\end{cases}
\end{align}
\end{lemma}
\begin{proof}
Denoting $\bar{v}(x,y,t)=e^{x+(\frac{\nu^2}{2}+\mu-r)t+\varepsilon}$, it is easy to see that 
\[v^\varepsilon(x,y,0)=v_0^\varepsilon(x)\leq e^{x+\varepsilon}=\bar{v}(x,y,0),\]
and 
\begin{align}
\mathcal B[\bar{v}]+\beta(\frac{1+\varepsilon-\bar v}{\varepsilon})&=(\frac{v^2}{2}+\mu-r)e^{x+(\frac{\nu^2}{2}+\mu-r)t+\varepsilon}-(\mu-r)e^{x+(\frac{\nu^2}{2}+\mu-r)t+\varepsilon}\notag\\
&\quad-\frac{\nu^2}{2}e^{x+(\frac{\nu^2}{2}+\mu-r)t+\varepsilon}+\frac{s}{m}e^{2[x+(\frac{\nu^2}{2}+\mu-r)t+\varepsilon]}+\beta(\frac{1+\varepsilon-\bar{v}}{\varepsilon})\geq 0.\notag
\end{align}
It follows by maximum principle that $v^\varepsilon\leq e^{x+(\frac{\nu^2}{2}+\mu-r)t+\varepsilon}$, hence we prove the second inequality of the lemma.

For the first inequality, consider the function
\[\underline{v}(x,y,t)=\rho\Big(\frac{\delta xe^{-\kappa t}}{a+e^y}\Big).\]
From the definition of $\rho(z)$, it holds that  
\begin{align}\label{qq}
\begin{cases}
\rho(z)=0\quad \rm for \quad z\le1,\\
\rho(z)=1\quad \rm for \quad z\ge3,\\
0<\rho'(z)<1\quad \rm for \quad 1<z<3,\\
\rho''>0 \quad \rm for \quad 1<z<2,\\
\rho''<0\quad \rm for\quad 2<z<3,\\
\rho(2)=\frac{1}{2},\rho''(3)=0.
\end{cases}
\end{align}
Thus, we have  that $\underline{v}(x,y,0)\leq v^\varepsilon(x,y,0)$.
For simplicity of notation, set $z=\frac{\delta xe^{-\kappa t}}{a+e^y}$ and one can compute that 
\[\underline{v}_t=-\kappa z\rho'(z),\]
\[\underline{v}_y=-\frac{\delta xe^{-\kappa t}e^y}{(a+e^y)^2}\rho'(z)=-\frac{e^y}{a+e^y}z\rho'(z),\]
\begin{align}
\underline{v}_{yy}&=[\frac{\delta xe^{-\kappa t}e^y}{(a+e^y)^2}]^2\rho''(z)+(e^y-a)\frac{\delta xe^{-\kappa t}}{(a+e^y)^3}e^y\rho'(z)\notag\\
&=\left(\frac{e^y}{a+e^y}\right)^2z^2\rho''(z)+\frac{e^y-a}{(a+e^y)^2}e^yz\rho'(z),\notag
\end{align}
\[\underline{v}_x=\frac{\delta e^{-\kappa t}}{a+e^y}\rho'(z),\]
\[\underline{v}_{xx}=\left(\frac{\delta e^{-\kappa t}}{a+e^y}\right)^2\rho''(z),\]
and
\[\beta(\frac{1+\varepsilon-\underline{v}}{\varepsilon})=\beta(1+\frac{1-\underline{v}}{\varepsilon})=0,\]
as $\underline v \le 1$. Hence,
\begin{align}
\mathcal B[\underline{v}]&=-\kappa z\rho'(z)-\frac{\sigma^2}{2}\left( \left(\frac{e^y}{a+e^y}\right)^2z^2\rho''(z)+\frac{e^y-a}{(a+e^y)^2}e^yz\rho'(z)\right)-\frac{\nu^2}{2}\Big(\frac{\delta e^{-\kappa t}}{a+e^y}\Big)^2\rho''(z)\notag\\
&\quad+(\mu+\frac12\sigma^2)\frac{e^y}{a+e^y}z\rho'(z)-(\mu-r)\rho(z)+\frac{e^y}{m}\frac{\delta e^{-\kappa t}}{a+e^y} \rho(z)\rho'(z)\notag\\
&= -\left( \frac{\sigma^2}{2}\left(\frac{e^y}{a+e^y}\right)^2z^2+\frac{\nu^2}{2}\Big(\frac{\delta e^{-\kappa t}}{a+e^y}\Big)^2\right)\rho''(z)\notag\\
&\quad -\left(\kappa+\frac{\sigma^2}{2}\frac{e^y-a}{(a+e^y)^2}e^y-(\mu+\frac{1}{2}\sigma^2) \frac{e^y}{a+e^y}  \right)z\rho'(z)-\left( \mu-r-\frac{e^y}{m}\frac{\delta e^{-\kappa t}}{a+e^y}\rho'(z)\right)\rho(z).
\end{align}
Since $\rho'(z)<1$ and the choice of $\delta$, we get that 
$$
\mu-r-\frac{e^y}{m}\frac{\delta e^{-\kappa t}}{a+e^y}\rho'(z) \ge \mu-r-\frac{\delta}{m} \ge  \frac{\mu-r}{2}.
$$
Similarly, if $\kappa>2(\sigma^2+\mu)$, 
$$
\kappa+\frac{\sigma^2}{2}\frac{e^y-a}{(a+e^y)^2}e^y-(\mu+\frac{1}{2}\sigma^2) \ge \kappa-\frac{\sigma^2}{2}-(\mu+\frac12 \sigma^2) \ge \frac \kappa 2 .
$$
Hence, when $z\leq2$, using $\rho''(z)\geq 0$, we obtain that 
\[\mathcal B[\underline{v}]\leq0.\]
In the case $z\geq 2$, we use $\rho\geq\frac{1}{2}$ to obtain that 
\[\mathcal{B}[\underline{v}]\leq-\frac{\kappa}{2}z\rho'(z)+(\frac{9\sigma^2}{2}+\frac{\delta^2\nu^2}{2a^2})|\rho''(z)|-\frac{\mu-r}{4}.\]
Since $\rho''(z)=0$ at $z=3$, one can find a small interval $[3-\eta,3]$ of $z$ such that $(\frac{9\sigma^2}{2}+\frac{\delta^2\nu^2}{2a^2})|\rho''(z)|-\frac{\mu-r}{4}\le 0$. For $z \in [2,3-\eta]$, $\rho'(z)$ is strictly above $0$. As $|\rho''(z)| \le 1$, one can choose $\kappa$ large enough to obtain that  
$$
-\frac{\kappa}{2}z\rho'(z)+(\frac{9\sigma^2}{2}+\frac{\delta^2\nu^2}{2a^2})|\rho''(z)| \le 0.
$$
Add the case $z\ge 3$, we get that $\mathcal B[\underline{v}] \le 0$. Then, by comparison principle, we have
\[v^\varepsilon(x,y,t)\geq \rho(\frac{\delta xe^{-\kappa t}}{a+e^y}).\]
\end{proof}

\begin{lemma}\label{lemma5.7}
For each $y \in \mathbb R$, $t>0$, $\varepsilon\in(0,1]$,
\[\lim\limits_{x\rightarrow+\infty} v^\varepsilon(x,y,t)\geq1,\quad \lim\limits_{x\rightarrow-\infty} v^\varepsilon(x,y,t)=0.\]
\end{lemma}
\begin{proof}
This follows from Lemmas \ref{lemma5.1} and \ref{lemma5.6}.
\end{proof}
\subsection{The Limit of $v^\varepsilon$}
Based on the  above prior estimates of the approximate solution $v^\varepsilon$,  we will next prove that the limit while $\varepsilon\rightarrow0$ is a solution of the obstacle problem (\ref{10}). In Lemma \ref{lemma5.1}, we prove that  $v^\varepsilon$ and $\beta(\frac{1+\varepsilon-v^\varepsilon}{\varepsilon})$ are uniformly bounded. From interior estimates for linear parabolic equation, one can get that, for any $p>1$ and  $K>0$, 
\begin{align}
||v^\varepsilon||_{W_p^{2,1}([-K,K]\times[-K,K]\times[\frac{1}{K},K))}\leq C(K,p).\label{*}
\end{align}
This implies that, taking sub-sequence if necessary, $v^\varepsilon$ converges locally uniformly to some function $v$, and the derivatives of $v^\varepsilon$ weakly converges to those of $v$. This implies that \[||v||_{W_p^{2,1}([-K,K]\times[\frac{1}{K},K]\times[\frac{1}{K},K])}\leq C(K,p).\]. 

From Lemma \ref{lemma5.7} and the fact that $\frac{\partial v^\varepsilon}{\partial x}>0$, there exists  $x^\varepsilon(\cdot,\cdot)$ such that for every $y\in \mathbb R$, $t\geq 0$,
\[v^\varepsilon(x,y,t)<1-\varepsilon \quad \rm if \quad x<x^\varepsilon(y,t),\]
\[1-\varepsilon<v^\varepsilon(x,y,t)<1+\varepsilon\quad \mbox{if} \quad x>x^\varepsilon(y,t),\]
\[v^\varepsilon(x^\varepsilon(y,t),y,t)=1-\varepsilon,\]
Moreover,
$x^\varepsilon\in C^\infty(\mathbb R\times[0,\infty))$ by the implicit function theorem.
Noting that $\frac{\partial v^\varepsilon}{\partial \varepsilon}\geq 0$ in Lemma \ref{lem_mono}, we deduce that for $0<\varepsilon_1<\varepsilon_2$, we have
\[x^{\varepsilon_1}(y,t)\ge x^{\varepsilon_2}(y,t)\quad \forall y\in \mathbb R,t>0.\]
Moreover, from Lemma \ref{lemma5.6}, we have that 
$$
-(\frac{\nu^2}{2}+\mu-r)t-\varepsilon+\log(1-\varepsilon)
\le x^\varepsilon(y,t) \le \frac{3(a+e^y)}{\delta}e^{\kappa t}.
$$
This implies that, for fixed $(y,t)$, $x^\varepsilon(y,t)$ admits a limit when $\varepsilon$ goes to $0$. Define
$$
x(y,t):=\lim\limits_{\varepsilon\rightarrow 0}x^\varepsilon(y,t) \quad \forall y\in\mathbb{R},t\geq 0.
$$
We have the following results for the existence of the solution,
\begin{theorem}\label{thm_exist}
$v \in W_{p,loc}^{2.1}(Q)$ is a solution of (\ref{10}) and $x(y,t)$ is the free boundary. 
\end{theorem}
\begin{proof}{\bf Setp 1.}
First, let us prove that $v$ satisfies the equation. Fix $(x_0,y_0,t_0)\in Q$.
 Consider two cases:

(1) $x_0<x(y_0,t_0)$;

(2) $x_0>x(y_0,t_0)$.

{\it Case 1: $x_0<x(y_0,t_0)$.} There exists $\varepsilon_0\in (0,1]$ such that
\[x^{\varepsilon_0}(y_0,t_0)>x_0.\]
By continuity of $x^{\varepsilon_0}$, there exists $\delta>0$, such that
\[x^{\varepsilon_0}(y,t)>x_0+\delta \quad \forall y\in (y_0-\delta,y_0+\delta), t\in(t_0-\delta,t_0+\delta).\]
By monotonicity, for $\varepsilon\in(0,\varepsilon_0)$, we have
\[x^{\varepsilon}(y,t)>x^{\varepsilon_0}(y,t)>x_0+\delta\quad \forall y\in (y_0-\delta,y_0+\delta), t\in(t_0-\delta,t_0+\delta).\]
Noting that $v^\varepsilon<1-\varepsilon$ when $x<x^\varepsilon(y,t)$, we see that
\[\beta(\frac{1+\varepsilon-v^\varepsilon}{\varepsilon})=0\quad \mbox{if} \quad x<x^\varepsilon(y,t).\]
Thus, for $\varepsilon\in(0,\varepsilon_0)$
\[\mathcal B[v^\varepsilon]=0 \quad \mbox{in} \quad (-\infty,x_0+\delta)\times(y_0-\delta,y_0+\delta)\times(t_0-\delta,t_0+\delta).\]
Sending $\varepsilon\rightarrow 0$, we obtain that 
\[\mathcal B[v]=0\quad \mbox{in} \quad (-\infty,x_0+\delta)\times(y_0-\delta,y_0+\delta)\times(t_0-\delta,t_0+\delta).\]
In addition,
\[v<v^\varepsilon<1\quad \mbox{in} \quad (-\infty,x_0+\delta)\times(y_0-\delta,y_0+\delta)\times(t_0-\delta,t_0+\delta).\]

{\it Case 2: $x_0>x(y_0,t_0)$.} Since $x^\varepsilon$ is increasing with respect to $\varepsilon$,
this implies that
\[x^\varepsilon(y_0,t_0)\le x(y_0,t_0)<x_0-\delta\quad\forall\varepsilon\in(0,1].\]
Thus,
\[1-\varepsilon<v^\varepsilon(x,y_0,t_0)<1+\varepsilon \quad\forall x\geq x_0.\]
Consequently,
\[v(x,y_0,t_0)=1 \quad\forall x\ge x_0.\]
We also have that 
\[\mathcal B[v^\varepsilon]=-\beta(\frac{1+\varepsilon-v^\varepsilon}{\varepsilon})\leq 0.\]
Sending $\varepsilon\rightarrow 0$, we have
\[B[v]\leq 0\quad \mbox{in}\quad L_{loc}^p(Q).\]
Thus, combining both cases, we obtain
\[\max\{B[v],v-1\}=0\quad \mbox{in} \quad L_{loc}^p(Q).\]

{\bf Step 2.} We prove that $v$ satisfies the initial condition \eqref{intial_cond}. Fixing $\iota>0$, let ${\underline \varphi}^\varepsilon$ and $\overline{\varphi}^\varepsilon$ respectively satisfies the following PDEs
$$
\mathcal L[\underline \varphi^\varepsilon]+\frac{e^y}{m}(1+\iota)\underline \varphi_x^\varepsilon=0,\quad \underline \varphi^\varepsilon(x,y,0)=v_0^\varepsilon,
$$
and
$$
\mathcal L[\overline \varphi^\varepsilon]=0,\quad\overline \varphi^\varepsilon(x,y,0)=v_0^\varepsilon.
$$
Similarly as before, one can show that $\underline{\varphi}^\varepsilon_x$, $\overline{\varphi}^\varepsilon_x \ge 0$. Denote by $w=v^\varepsilon-\underline{\varphi}^\varepsilon+\beta(0)t$. It holds that 
$$
\mathcal L[w]+\frac{e^y}{m}v^\varepsilon w_x= \frac{e^y}{m}(1+\iota-v^\varepsilon)\underline{\varphi}^\varepsilon_x+\beta(0)-\beta(\frac{1+\varepsilon-v^\varepsilon}{\varepsilon}).
$$
Since we have proved that $v^\varepsilon\le 1+\varepsilon$, we see that, for $\varepsilon<\iota$, 
$$
\mathcal L[w]+\frac{e^y}{m}v^\varepsilon w_x \ge 0.
$$
From maximum principle, we get that 
$v^\varepsilon \ge \underline \varphi^\varepsilon-\beta(0)t$. With a similar argument, we will have that 
\begin{equation}\label{ineq_varphi}
    \underline \varphi^\varepsilon -\beta(0)t \le v^\varepsilon \le \overline{\varphi}^\varepsilon.
\end{equation}

Now, let us give a probabilistic representation for $\underline \varphi^\varepsilon$. Let $Y_s=y+(\mu+\frac12\sigma^2)s+\sigma W^1_s$ and $X_s$ satisfies
$$
dX_s=-\frac{e^{Y_t}}{m}(1+\iota) ds+\nu dW^2_s, \quad X_0=x,
$$
where $W^1$ and $W^2$ are two independent scalar-valued Brownian motion. From Feymann-Kac representation, it holds that 
$$
\underline {\varphi}^\varepsilon(x,y,t)=e^{(\mu-r)t}\mathbb E\left[v_0^{\varepsilon}(X_t)\right]. 
$$
Letting $\varepsilon \rightarrow 0$ and using dominate convergence theorem, we get that 
$$
\lim_{\varepsilon\rightarrow 0} \underline \varphi^\varepsilon(x,y,t)=e^{(\mu-r)t}\mathbb E\left[{\bf 1}_{\{X_t>0\}}\right].
$$
Define a stopping time $\tau$ as 
$$
\tau=\inf\{s\ge 0|(\mu+\frac12\sigma^2)s+\sigma W_s^1\ge 1\}.
$$
It is easy to get that $P(\tau < t)\rightarrow 0$ as $t$ approaches $0$. On the set $\{\tau \ge t\}$, 
$$
X_t\ge \underline X_t:=x-\frac{e^{y+1}}{m}(1+\iota)t+\nu W^2_t,
$$
which implies that 
\begin{equation*}
\begin{split}
&\mathbb E\left[{\bf 1}_{\{X_t>0\}}\right]\ge \mathbb E\left[{\bf 1}_{\{X_t>0\}}{\bf 1}_{\{\tau \ge t\}}\right] \ge \mathbb E\left[{\bf 1}_{\{\underline X_t>0\}}{\bf 1}_{\{\tau \ge t\}}\right] \\=& \mathbb E\left[{\bf 1}_{\{\underline X_t>0\}}\right]-\mathbb E\left[{\bf 1}_{\{\underline X_t>0\}}{\bf 1}_{\{\tau < t\}}\right] \ge \mathbb E\left[{\bf 1}_{\{\underline X_t>0\}}\right]-P(\tau<t).
\end{split}
\end{equation*}
One can compute that 
$$
\mathbb E\left[H(\underline X_t)\right]=P(\underline X_t \ge 0)=1-\varphi(\frac{\frac{e^{y+1}}{m}(1+\iota)t-x}{\nu \sqrt{t}})=\varphi(\frac{x-\frac{e^{y+1}}{m}(1+\iota)t}{\nu \sqrt{t}}).
$$
Following a similar argument,
$$
\lim_{\varepsilon\rightarrow 0}\overline \varphi^\epsilon(x,y,t)=\mathbb E\left[{\bf 1}_{\{\overline X_t>0\}}\right],
$$
where $\overline X_t=x+\nu W_t^2$. An elementary computation gives us that 
$$
\mathbb E\left[{\bf 1}_{\{\overline X_t>0\}}\right]=\varphi(\frac{x}{\nu \sqrt{t}}).
$$
Thus, letting $\varepsilon \rightarrow 0$ in \eqref{ineq_varphi} , we have that 
$$
e^{(\mu-r)t}\left\{\varphi(\frac{x-\frac{e^{y+1}}{m}(1+\iota)t}{\nu \sqrt{t}})-\beta(0)t-P(\tau<t)\right\} \le v(x,y,t)\le e^{(\mu-r)t}\varphi(\frac{x}{\nu \sqrt{t}}),
$$
which implies that, for any $K>0$,
$$
\lim_{t\rightarrow 0} \sup_{|y|\le K,x\in \mathbb R}|v(x,y,t)-\varphi(\frac{x}{\nu \sqrt{t}})|=0.
$$

\end{proof}
Then, we have the existence of the solution for \eqref{888}.
\begin{theorem}
There exists $\Phi\in C(\bar{Q})\cap C^{2,1}(Q)$ be the solution of (\ref{888}), which satisfies the growth condition \eqref{growth_condi}. 
\end{theorem}
\begin{proof}
Define 
$$
u(x,y,t)=\int_{-\infty}^x v(z,y,t)dz.
$$
Letting $\varepsilon \rightarrow 0$ for the inequality in Lemma \ref{lemma5.6}, we prove that 
$$
v(x,y,t)\le e^{x+(\frac{\nu^2}{2}+\mu-r)t}.
$$
This implies that $u$ is well-defined and $u_x=v$.  Fix $y_0$ and $t_0>0$. For $x<-2(\frac{\nu^2}{2}+\mu-r)t_0,y\in \mathbb R$ and $t\in[\frac{t_0}{2},2t_0]$, we see that $v$ satisfies $\mathcal B[v]=0$. Using interior $L^p$ estimate, we get that 
$$
||v^\varepsilon||_{W_p^{3,\frac{3}{2}}([-x-1,x+1]\times[-K,K]\times[\frac{t_0}{2},2t_0)]}\leq C(K,p)e^{ny} e^x,
$$
for some constant $n$. This implies that the derivatives of $v$ goes to $0$ with a rate $e^x$ when $x$ goes to $-\infty$. Thus, one can interchange the order of integration and differentialation, i.e. for example,
$$
\frac{\partial u}{\partial t}= \int_{-\infty}^x \frac{\partial v}{\partial t}(z,y,t)dz. 
$$
Then, we have
$$
\mathcal B[v]=\frac{\partial }{\partial x}\mathcal A[u],
$$
or equivalently,
$$
\mathcal A[u]=\int_{-\infty}^x \mathcal B[v](z,y,t)dz.
$$
Hence, one can verify that 
$$
\max\{\mathcal A[u],\frac{\partial u}{\partial x}-1\}=0.
$$
Moreover, from the initial condition \eqref{intial_cond} of $v$, we also have that 
$$
\lim_{t\rightarrow0}u(x,y,t)=\lim_{t\rightarrow0} \int_{-\infty}^x v(z,y,t)dz=\int_{-\infty}^x \mathbf{1}_{z>0}dz=x^+.
$$
Hence, $u$ is a solution of \eqref{9}. Define $\Phi(x,s,t)=su(x,\log s,t)$. It is obvious that $\Phi$ solves \eqref{888}. Finally, let us check that $\Phi$ satisfies the growth condition \eqref{growth_condi}. Since $\frac{\partial \Phi}{\partial x}=s\frac{\partial u}{\partial x}$, it is obvious that  $|\frac{\partial \Phi}{\partial x}|\le s$. For $\frac{\partial \Phi}{\partial s}$, one can get that
$$
\frac{\partial \Phi}{\partial s}(x,s,t)=u(x,\log s,t)+\frac{\partial u}{\partial y}(x,\log s,t).
$$
By the definition of $u$, we have that
$$
\frac{\partial u}{\partial y}=\int_{-\infty}^x \frac{\partial v}{\partial y}(z,y,t)dz.  
$$
Since $x(y,t) \le 3\frac{a+e^y}{\delta}e^{\kappa t}$, it holds that, for $x >3\frac{a+e^y}{\delta}e^{\kappa T}+1$ 
$$
\frac{\partial v}{\partial y}(x,y,t)=0.
$$
Previous argument has proved that, for $x\le -(\frac{\nu^2}{2}+\mu-r)T$, 
$$
|\frac{\partial v}{\partial y}(x,y,t)| \le C(1+e^{ny}) e^x. 
$$
In general, due to the boundedness of $v$, we have 
$$
|\frac{\partial v}{\partial y}(x,y,t)|\le C(1+e^{ny}).
$$
Thus,
\begin{equation*}
\begin{split}
|\frac{\partial u}{\partial y}(x,y,t)| &\le \int_{-\infty}^{3\frac{a+e^y}{\delta}e^{\kappa T}+1} |\frac{\partial u}{\partial y}(z,y,t)|dz\\
&\le \int_{-\infty}^{-(\frac{\nu^2}{2}+\mu-r)T} |\frac{\partial u}{\partial y}(z,y,t)|dz+\int_{-(\frac{\nu^2}{2}+\mu-r)T} ^{3\frac{a+e^y}{\delta}e^{\kappa T}+1}|\frac{\partial u}{\partial y}(z,y,t)|dz\\
&\le Ce^{ny}\int_{-\infty}^{-(\frac{\nu^2}{2}+\mu-r)T} e^x+Ce^{ny} \int_{-(\frac{\nu^2}{2}+\mu-r)T} ^{3\frac{a+e^y}{\delta}e^{\kappa T}+1}dz\\
&=Ce^{ny}e^{-(\frac{\nu^2}{2}+\mu-r)T}+Ce^{ny}\left(3\frac{a+e^y}{\delta}e^{\kappa T}+1-(\frac{\nu^2}{2}+\mu-r)T\right)\\
&\le C(1+e^{(n+1)y}) \le C(e^{-(n+1)y}+e^{(n+1)y}).
\end{split}
\end{equation*}
This gives the growth condition for $\Phi$.
\end{proof}

\begin{remark}
Differentiating $v^\varepsilon(x^\varepsilon,y,t)-1+\varepsilon=0$, we obtain
$\frac{\partial x^\varepsilon}{\partial y}=-\frac{\partial v^\varepsilon/\partial y}{\partial v^\varepsilon/\partial x^\varepsilon}>0$.
Then,
$\frac{\partial x}{\partial y}\geq0$, i.e., the free boundary increases with the price of carbon allowances. This implies that with a higher price for allowances, the enterprise is less inclined to diminish the excess emissions by purchasing additional quotas.
\end{remark}

\section{Uniqueness of Variation Inequality} \label{sect_uniq}
In this section, we will prove the uniqueness of solutions  to (\ref{10}) and (\ref{9}) by establishing an appropriate comparison principle. For notational simplicity, we introduce the operator $\mathcal F$ as 
\[\mathcal F[\Phi]=\max\Big\{\mathcal{L}\Phi+\frac{e^y}{m}\Phi\frac{\partial \Phi}{\partial x},\Phi-1\Big\},\]

\begin{lemma}[Comparison Principle]\label{theorem6.1}
Assume that $v_1$ and $v_2$ in $L^\infty(Q)\cap(\cap_{p>1}W_{p,loc}^{2,1}(Q))$ satisfy 
the initial condition \eqref{intial_cond} and 
\begin{align}
F[v_1]\geq F[v_2]\quad a.e. \quad \rm in\quad Q.
\end{align}
We also assume the validity of the following technical condition
\[\max\Big\{\frac{\partial v_1}{\partial x},\frac{\partial v_2}{\partial x}\Big\}\geq 0\quad \rm in\quad Q.\]
Then $v_1\geq v_2$ \rm in $Q$.
\end{lemma}
\begin{proof}
Assume that the assertion is not true. Then there exists $x_0,y_0\in\mathbb{R}$ and $t_0\geq 0$ such that $v_1(x_0,y_0,t_0)<v_2(x_0,y_0,t_0)$. For a positive constant $\kappa$ to be defined later, we define
\[\eta:=\sup_{(x,y)\in\Omega,t\in[0,t_0]}[v_2(x,y,t)-v_1(x,y,t)]e^{-\kappa t}.\]
Let $(x_1,y_1,t_1)\in \Omega\times[0,T]$ be a point such that
\[(v_2(x_1,y_1,t_1)-v_1(x_1,y_1,t_1))e^{-\kappa t_1}\geq \frac{7}{8}\eta.\]
Let $\xi\in C^\infty(\mathbb{R})$ be a function satisfying
\[\xi=0\quad \mbox{on}\quad[-1,1],\quad \xi=1\quad \mbox{on}\quad(-\infty,-4)\cup(4,\infty),\quad |\xi'|\leq1,\quad |\xi''|\leq1.\]
Set
\[\varphi(x,y):= \frac{1}{2}(e^{y-y_1}+e^{y_1-y})+\xi(\frac{x-x_1}{e^{y_1}+1}),\]
and
\[M:=\max\{||v_1||_{L^\infty},||v_2||_{L^\infty}\}.\]
This function has the following property:

(1) $\varphi(x_1,y_1)=1=\min_{\Omega}\varphi$;

(2) Set $D=[x_1-4(1+e^{y_1}),x_1+4(1+e^{y_1})]\times[y_1-\log 4,y_1+\log 4]$. Then
\[\varphi\geq 2 \quad \mbox{for} \quad (x,y)\notin D.\]
Setting $\psi(x,y,t)=\varphi(x,y)e^{\kappa t}$, we calculate, for $i=1,2$,
\begin{align}
\hat{\mathcal{L}}_i\psi&:=\frac{\partial \psi}{\partial t}-\frac{\sigma^2}{2}\frac{\partial^2\psi}{\partial y^2}-(\frac12\sigma^2+\mu)\frac{\partial\psi}{\partial y}-\frac{\nu^2}{2}\frac{\partial^2\psi}{\partial x^2}-(\mu-r)\psi+\frac{e^yv_i}{m}\psi_x\notag\\
&=e^{\kappa t}\bigg\{\kappa \varphi -\frac{\sigma^2}{2}\left(\frac12e^{y-y_1}+\frac12e^{y_1-y}\right)-(\frac12 \sigma^2+\mu)\left(\frac12e^{y-y_1}-\frac12e^{y_1-y}\right)\notag\\
&\quad-\frac{\nu^2}{2}\left(\frac{1}{1+e^{y_1}}\right)^2\xi''(\frac{x-x_1}{1+e^{y_1}})-(\mu-r) \varphi +\frac{e^{y}}{m} v_i\frac{1}{1+e^{y_1}}\xi'(\frac{x-x_1}{1+e^{y_1}})
\bigg\}\notag\\
&=e^{\kappa t}\bigg\{ \frac12 e^{y-y_1}\left(\kappa-\sigma^2-\mu-(\mu-r)+\frac{2e^{y_1}v_i}{(1+e^{y_1}m)}\xi'(\frac{x-x_1}{1+e^{y_1}})\right)\notag\\
&\quad+\frac12 e^{y_1-y}\left(\kappa+\mu-(\mu-r) 
 \right)+\left(\kappa-(\mu-r)\right)\xi(\frac{x-x_1}{1+e^{y_1}})-\frac{\nu^2}{2}\left(\frac{1}{1+e^{y_1}}\right)^2\xi''(\frac{x-x_1}{1+e^{y_1}})
\bigg\}\notag\\
\end{align}
Assuming that
\[\kappa=\sigma^2+2\mu-r+\frac{2}{m}(||v_1||_{L^\infty}+||v_2||_{L^\infty})+\frac{\nu^2}{2}+2,\]
we get that 
\begin{align}
\hat{\mathcal{L}}_i\psi&\geq e^{\kappa t}\Big[\Big(2+\frac{\nu^2}{2}\Big)\Big(\frac12e^{y-y_1}+\frac12 e^{y_1-y}\Big)+\xi(\frac{x-x_1}{1+e^{y_1}})-\frac{\nu^2}{2}\Big]\notag\\
&\geq e^{\kappa t}\Big[\Big(\frac12e^{y-y_1}+\frac12 e^{y_1-y}\Big)+\xi(\frac{x-x_1}{e^{y_1}+1})\Big]=\psi.\notag
\end{align}
Now define
\[\hat{\eta}=\sup_{\Omega\times[0,t_0]}\frac{(v_2-v_1)e^{-\kappa t}}{\varphi}=\sup_{\Omega\times[0,t_0]}\frac{(v_2-v_1)}{\psi}.\]
Evaluating at $(x_1,y_1,t_1)$ and using $\varphi(x_1,y_1)=1$, we see that $\hat{\eta}\geq\frac{7}{8}\eta$.

When $(x,y)\notin D$, we use $\varphi(x,y)\geq 2$ to obtain that
\begin{align}
\frac{(v_2-v_1)e^{-\kappa t}}{\varphi(x,y)}\leq\frac{\eta}{2}\quad \mbox{holds for} 
\quad t\in[0,t_0],\quad (x,y)\notin D.\label{aa}
\end{align}
From the boundary condition, we have 
that $v_2-v_1$ locally uniformly convergences to $0$ when $t\rightarrow 0$. Thus, there exists $\hat{t}_0\in(0,t_0)$ such that
\begin{align}
\frac{(v_2-v_1)e^{-\kappa t}}{\varphi(x,y)}\leq\frac{\eta}{2}\quad \mbox{also  holds for} \quad t\in[0,\hat{t}_0],\quad (x,y)\in D.\label{bb}
\end{align}
Hence there exists $p_2=(x_2,y_2,t_2)\in D\times(\hat{t_0},t_0)$ such that
\[\sup_{\Omega\times[0,t_0]}\frac{v_2-v_1}{\psi}=\hat{\eta}=\frac{v_2-v_1}{\psi}|_{p_2}.\]

First, we consider the case that $v_1$ and $v_2$ are $C^{2,1}$ near $p_2$. Note that we have
\[\max_{\Omega\times[0,t_0]}\{v_2-v_1-\hat{\eta}\psi\}=0=v_2(p_2)-v_1(p_2)-\hat{\eta}\psi(p_2).\]
Thus,
\[\frac{\partial}{\partial t}(v_2-v_1-\hat{\eta}\psi)(p_2)\geq0,\]
\[\frac{\partial}{\partial x}(v_2-v_1-\hat{\eta}\psi)(p_2)=0,\quad \frac{\partial}{\partial y}(v_2-v_1-\hat{\eta}\psi)(p_2)=0,\]
\[-\frac{\partial^2}{\partial x^2}(v_2-v_1-\hat{\eta}\psi)(p_2)\geq0,\quad -\frac{\partial^2}{\partial y^2}(v_2-v_1-\hat{\eta}\psi)(p_2)\geq0,\]
and
\[(v_2(p_2)-1)-(v_1(p_2)-1)=\hat{\eta}\psi(p_2)\geq\hat{\eta}.\]
Then, 
\begin{align}
&\quad\mathcal{B}[v_2]-\mathcal{B}[v_1]|_{p_2}\notag\\
&\geq\mathcal{B}[v_1+\hat{\eta}\psi]-\mathcal{B}[v_1]|_{p_2}\notag\\
&=\mathcal{L}v_1|_{p_2}+\hat{\eta}\mathcal{L}\psi|_{p_2}+\frac{e^y}{m}[v_1(p_2)+\hat{\eta}\psi(p_2)][\frac{\partial v_1}{\partial x}(p_2)+\hat{\eta}\frac{\partial\psi}{\partial x}(p_2)]-\mathcal{L}v_1|_{p_2}-\frac{e^y}{m}v_1(p_2)\frac{\partial v_1}{\partial x}(p_2)\notag\\
&=\hat{\eta}\mathcal{L}\psi|_{p_2}+\frac{e^y}{m}v_1(p_2)\hat{\eta}\frac{\partial \psi}{\partial x}(p_2)+\frac{e^y}{m}\hat{\eta}\psi(p_2)\frac{\partial v_1}{\partial x}(p_2)+\frac{e^y}{m}\hat{\eta}^2\psi(p_2)\frac{\partial \psi}{\partial x}(p_2)\notag\\
&=\hat{\eta}\Big[\mathcal{L}\psi+\frac{e^y}{m}\Big(v_2\frac{\partial\psi}{\partial x}+\psi\frac{\partial v_1}{\partial x}\Big)\Big]\Big|_{p_2}.
\end{align}

Consider two cases: $\frac{\partial v_1}{\partial x}(p_2)\geq 0$,  and   $\frac{\partial v_2}{\partial x}(p_2)\geq 0$. In the first case, we have
\[\mathcal{B}[v_2]-\mathcal{B}[v_1]|_{p_2}\geq\hat{\eta}[\hat{\mathcal{L}}_2\psi](p_2)\geq\hat{\eta}\psi(p_2)\geq\hat{\eta}.\]
In the second case, we use
\[\Big[\frac{\partial v_1}{\partial x}\psi+v_2\frac{\partial \psi}{\partial x}\Big]\Big|_{p_2}=\Big[\frac{\partial v_2}{\partial x}\psi+v_1\frac{\partial \psi}{\partial x}\Big]\Big|_{p_2}\]
to obtain
\[\mathcal{B}[v_2](p_2)-\mathcal{B}[v_1](p_2)\geq\hat{\eta}[\hat{\mathcal{L}}_1\psi](p_2)\geq\hat{\eta}\psi(p_2)\geq\hat{\eta}.\]
In conclusion, we obtain that 
\[[F[v_2]-F[v_1]]|_{p_2}=[\max\{\mathcal{B}[v_2],v_2-1\}-\max\{\mathcal{B}[v_1],v_1-1\}]|_{p_2}\geq \hat{\eta},\]
which contradicts the assumption.

Next, we drop the assumption $v_i\in C^{2,1}$, only assuming that $v_i\in W_{p,loc}^{2,1}$ for some $p\gg1$. Let $\rho \in C^\infty(\mathbb{R}^3)$ satisfies
\[\rho\geq 0\quad \mbox{in} \quad\mathbb{R}^3,\quad \iiint_{\mathbb{R}^3}\rho=1,\quad \rho(x,y,t)=0\quad \mbox{if}\quad x^2+y^2+t^2\geq1.\]
Set $\rho_\varepsilon(x,y,z)=\frac{1}{\varepsilon^3}\rho(\frac{x}{\varepsilon},\frac{y}{\varepsilon},\frac{z}{\varepsilon})$ and define
\[v_i^\varepsilon=\int_0^\infty\!\!\!\int_{-\infty}^\infty\!\!\int_{-\infty}^\infty\!\!\!\! v_i(\tilde{x},\tilde{y},\tilde{t})\rho_\varepsilon(x-\tilde{x},y-\tilde{y},t-\tilde{t})d\tilde{x}d\tilde{y}d\tilde{t}\in C^\infty(\mathbb{R}^3).\]
Note that when $t>\varepsilon$, we have
\begin{align}
&\quad(\frac{\partial v_i^\varepsilon}{\partial x},\frac{\partial^2 v_i^\varepsilon}{\partial x^2},\frac{\partial v_i^\varepsilon}{\partial y},\frac{\partial v_i^\varepsilon}{\partial t},\frac{\partial^2 v_i^\varepsilon}{\partial y^2})\notag\\
&=\int_0^\infty\int_{-\infty}^\infty\int_{-\infty}^\infty(\frac{\partial v_i}{\partial \tilde{x}},\frac{\partial^2 v_i}{\partial \tilde{x}^2},\frac{\partial v_i}{\partial \tilde{y}}, \frac{\partial v_i}{\partial \tilde{t}},\frac{\partial^2 v_i}{\partial \tilde{y}^2})(\tilde{x},\tilde{y},\tilde{t})\rho_\varepsilon(x-\tilde{x},y-\tilde{y},t-\tilde{t})d\tilde{x}d\tilde{y}d\tilde{t}
\end{align}
Define
\[\eta^\varepsilon=\sup_{\Omega\times[0,t_0]}\frac{(v_2^\varepsilon-v_1^\varepsilon)e^{-\kappa t}}{\psi_1(x,y,t)},\]
where
$\psi_1(x,y,t):=[\varphi(x,y)+(x-x_2)^4+(y-y_2)^4+(t-t_2)^4]e^{\kappa t}$.

In view of (\ref{aa}) and (\ref{bb}), we see that, when $\varepsilon$ is sufficiently small, there exists
\[p^\varepsilon=(x^\varepsilon,y^\varepsilon,t^\varepsilon)\in D\times(\hat t_0,t_0)\]
such that the maximum is obtained at $p^\varepsilon$. In addition,
\[\lim\limits_{\varepsilon\rightarrow 0}|p^\varepsilon-(x_2,y_2,t_2)|=0,\quad \lim\limits_{\varepsilon\rightarrow 0}\eta^\varepsilon=\hat{\eta}.\]
Since $\frac{\partial v_1}{\partial x}$, $\frac{\partial v_2}{\partial x}$ are continuous, we have
\[v_2(p_2)-v_1(p_2)=\lim\limits_{\varepsilon\rightarrow 0}[v_2(p^\varepsilon)-v_1(p^\varepsilon)]=\lim\limits_{\varepsilon\rightarrow 0}\eta^\varepsilon\psi(p^\varepsilon)=\hat{\eta}\psi(p_2)\geq \hat{\eta}.\]
Thus, $v_2(p)-1>v_1(p)-1$ in a neighborhood of $p_2$, and this implies that  $\mathcal{B}[v_1]-\mathcal{B}[v_2]\geq 0$ as $F[v_1]\geq F[v_2]$. Then, when $\varepsilon$ is sufficiently small,
\begin{align}
0&\leq\int_0^\infty\int_{-\infty}^\infty\int_{-\infty}^\infty\frac{1}{\varepsilon^3}\rho(\frac{p^\varepsilon-(\tilde{x},\tilde{y},\tilde{t})}{\varepsilon})(\mathcal{B}[v_1]-\mathcal{B}[v_2])d\tilde{x}d\tilde{y}d\tilde{t}\notag\\
&=\mathcal{B}[v_1^\varepsilon](p^\varepsilon)-\mathcal{B}[v_2^\varepsilon](p^\varepsilon)+\alpha_\varepsilon.\notag
\end{align}
where $\lim\limits_{\varepsilon\rightarrow0}\alpha_\varepsilon=0$.

This implies that
\begin{align}
0&\geq\mathcal{B}[v_2^\varepsilon](p^\varepsilon)-\mathcal{B}[v_1^\varepsilon](p^\varepsilon)-\alpha_\varepsilon\notag\\
&\geq\mathcal{B}[v_1^\varepsilon+\eta^\varepsilon\psi_1](p^\varepsilon)-\mathcal{B}[v_1^\varepsilon](p^\varepsilon)-\alpha_\varepsilon\notag\\
&=\eta^\varepsilon\Big[\mathcal{L}\psi_1+\frac{e^y}{m}\Big(v_2^\varepsilon\frac{\partial\psi_1}{\partial x}+\psi_1\frac{\partial v_1^\varepsilon}{\partial x}\Big)\Big]\Big|_{p^\varepsilon}-\alpha_\varepsilon\notag\\
&=\eta^\varepsilon\Big[\mathcal{L}\psi_1+\frac{e^y}{m}v_2\frac{\partial\psi_1}{\partial x}+\frac{e^y}{m}(v_2^\varepsilon-v_2)\frac{\partial\psi_1}{\partial x}+\frac{e^y}{m}\psi_1(\frac{\partial v_1^\varepsilon}{\partial x}-\frac{\partial v_1}{\partial x})+\frac{e^y}{m}\psi_1\frac{\partial v_1}{\partial x}\Big]\Big|_{p^\varepsilon}-\alpha_\varepsilon\notag\\
&= \eta^\varepsilon\Big[\hat{\mathcal{L}}_2\psi_1+\frac{e^y}{m}(v_2^\varepsilon-v_2)\frac{\partial \psi}{\partial x}+\frac{e^y}{m}\psi_1(\frac{\partial v_1^\varepsilon}{\partial x}-\frac{\partial v_1}{\partial x})+\frac{e^y}{m}\psi_1\frac{\partial v_1}{\partial x}\Big]\Big|_{p^\varepsilon}-\alpha_\varepsilon\notag\\
&=\eta^\varepsilon\Big[\hat{\mathcal{L}}_2\psi+4(t-t_2)^3+(\kappa-\mu+r)[(x-x_2)^4+(y-y_2)^4+(t-t_2)^4]\notag\\
&\quad-6\sigma^2(y-y_2)^2-4(\frac12\sigma^2+\mu)(y-y_2)^3-6\nu^2(x-x_2)^2+\frac{4e^yv_i}{m}(x-x_2)^3\notag\\
&\quad+\frac{e^y}{m}(v_2^\varepsilon-v_2)\frac{\partial \psi}{\partial x}+\frac{e^y}{m}\psi_1(\frac{\partial v_1^\varepsilon}{\partial x}-\frac{\partial v_1}{\partial x})+\frac{e^y}{m}\psi_1\frac{\partial v_1}{\partial x}\Big]\Big|_{p^\varepsilon}-\alpha_\varepsilon\notag\\
&\geq \eta^\varepsilon\Big[1+4(t-t_2)^3+(\kappa-\mu+r)[(x-x_2)^4+(y-y_2)^4+(t-t_2)^4]\notag\\
&\quad-6\sigma^2(y-y_2)^2-4(\frac12\sigma^2+\mu)(y-y_2)^3-6\nu^2(x-x_2)^2+\frac{4e^yv_i}{m}(x-x_2)^3\notag\\
&\quad+\frac{e^y}{m}(v_2^\varepsilon-v_2)\frac{\partial \psi}{\partial x}+\frac{e^y}{m}\psi_1(\frac{\partial v_1^\varepsilon}{\partial x}-\frac{\partial v_1}{\partial x})+\frac{e^y}{m}\psi_1\frac{\partial v_1}{\partial x}\Big]\Big|_{p^\varepsilon}-\alpha_\varepsilon.\notag
\end{align}
Sending $\varepsilon\rightarrow 0$, we then obtain $\hat{\eta}\leq 0$, which is a contradiction. Thus, we show that $v_2\leq v_1$ in $\Omega \times [0,\infty)$.
\end{proof}
\begin{theorem}[Uniqueness of $v$]\label{theorem6.2}
There exists a unique $v:Q\mapsto \mathbb{R}$ satisfying
\[ v \in L^\infty(Q)\cap W_{p,loc}^{2,1}(Q)\quad(p>4),\]
\begin{align}
\max\{\mathcal{B}[v],v-1\}=0\quad a.e. \quad in\quad Q,
\end{align}
\[\frac{\partial v}{\partial x}> 0\quad in\quad Q,\]
and the initial condition \eqref{intial_cond}.
\end{theorem}
\begin{proof}
The assertion follows by comparison principle (Lemma \ref{theorem6.1}) and Theorem \ref{thm_exist}.
\end{proof}

Theorem \ref{theorem6.2}  implies that the derivative of the value function is the unique solution of \eqref{10}. Hence, the trading boundary is fully characterized as it is the level set of derivative. For completeness, we also give the uniqueness of the solution of \eqref{9} with a growth condition.    \begin{theorem}[Uniqueness of $u$]
There exists a unique $u\in C(\bar{Q})\cap C^{2,1}(Q)$ satisfying
\begin{enumerate}[label={(\arabic*)}]
\item
\begin{align}
\quad\begin{cases}
\max\{\mathcal{A}[u],\frac{\partial u}{\partial x}-1\}=0\quad  in\quad Q,\\
u(x,y,0)=x^+\quad for \quad x,y\in\mathbb{R}.
\end{cases}
\end{align}
\item There exists a positive constant $C$ and positive integer $n$ such that
\[0\leq u(x,t)\leq C(e^{yn}+1)(1+|x|)e^{nt}.\]
\item $u\geq0$, $\frac{\partial u}{\partial x}\geq 0$, $\frac{\partial^2u}{\partial x^2}\geq 0$.
\end{enumerate}
\end{theorem}
\begin{proof}
Let $u$ be a generic solution that satisfies (i), (ii), (iii). Then there exists $x(y,t)$ such that
\[\frac{\partial u}{\partial x}=1 \quad \mbox{if} \quad x\geq x(y,t),\]
\[0\leq \frac{\partial u}{\partial x}<1\quad \mbox{if} \quad x<x(y,t).\]
Note that for $x<0$,
\[u(x,y,t)=u(0,y,t)-\int_x^0\frac{\partial u}{\partial z}(z,y,t)dz.\]
If $x(y,t)<0$, we have
\[0\leq u(x(y,t),y,t)\leq u(0,y,t)+x(y,t).\]
Thus,
\[x(y,t)\geq-u(0,y,t)\geq-C(1+e^{yn})e^{nt}.\]
It follows that
\begin{align}
\begin{cases}
\mathcal{A}[u]=0\quad  \mbox{for} \quad y \in \mathbb R,\quad t>0,\quad x<-C(1+e^{ny})e^{nt}\\
u(x,y,0)=0\quad\mbox{for}  \quad y\in \mathbb R,\quad x<0.
\end{cases}
\end{align}
This implies
\begin{align}
\lim\limits_{x\rightarrow-\infty}\sup_{y\leq M,0\leq t\leq T}u(x,y,t)=0\quad\forall M>0.\label{cc}
\end{align}

Now let $u_1$ and $u_2$ be two solutions. Assume that $u_1\not\equiv u_2$. Then there exists $p=(x_0,y_0,t_0)\in Q$ such that $u_2(p)\neq u_1(p)$. Assume that $u_2(p)>u_1(p)$. Now for a constant $\kappa$ to be defined later, and $N=n+1$, denote
\[\eta:=\sup_{\Omega\times[0,t_0]}\frac{u_2-u_1}{e^{\kappa t}[e^{Ny}+e^{-Ny}+e^x]}>0.\]
By the growth condition, there exists $M>0$, such that
\[\eta=\sup_{x\in\mathbb{R},y\in[\frac{1}{M},M],t\in[0,t_0]}\frac{u_2-u_1}{e^{\kappa t}[e^{Ny}+e^{-Ny}+e^x]}.\]
Finally, in view of (\ref{cc}), for $u=u_1$ and $u=u_2$, we find that there exists $p_1=(x_1,y_1,t_1)\in \mathbb{R}\times[\frac{1}{M},M]\times(0,t_0]$ such that the maximum is obtained at $p_1$.

Now define
\[\psi=e^{\kappa t}[e^{Ny}+e^{-Ny}+e^x].\]
Then we have
\[u_2\leq u_1+\eta\psi\quad \mbox{in} \quad \Omega\times[0,t_0],\]
\[u_2=u_1+\eta\psi\quad \mbox{at} \quad p_1.\]
Consequently,
\begin{align}
0&\geq [\mathcal{A}[u_1+\eta \psi]-\mathcal{A}[u_2]]|_{p_1}\notag\\
&=\Big[\mathcal{A}[u_1]-\mathcal{A}[u_2]+\eta\Big(\mathcal{L}\psi+\frac{e^y}{m}\frac{\partial u_1}{\partial x}\frac{\partial\psi}{\partial x}+\frac{e^y}{2m}\eta(\frac{\partial \psi}{\partial x})^2\Big)\Big]\Big|_{p_1}\notag\\
&\geq[\mathcal{A}[u_1]-\mathcal{A}[u_2]+\eta\mathcal{L}\psi]|_{p_1}.\notag
\end{align}
We can calculate that 
\begin{align}
\mathcal{L}\psi&=e^{Ny}e^{\kappa t}[\kappa-\frac{\sigma^2}{2}N^2-(\mu+\frac12\sigma^2)N-(\mu-r)]\notag\\
&\quad+e^{-Ny}e^{\kappa t}[\kappa-\frac{\sigma^2}{2}N^2+(\mu+\frac12\sigma^2)N-(\mu-r)]\notag\\
&\quad+e^{x+\kappa t}[\kappa-\frac{\nu^2}{2}-(\mu-r)]\geq\psi\geq1,\notag
\end{align}
if \[\kappa=1+\max\Big\{\frac{\sigma^2}{2}N^2+(\mu+\frac12\sigma^2)N+(\mu-r),\frac{\nu^2}{2}+\mu-r\Big\}.\]
This implies
\[\mathcal{A}[u_1](p_1)\leq\mathcal{A}[u_2](p_1)-\eta \psi(p_1)\leq\mathcal{A}[u_2](p_1)-\eta \frac{\partial \psi}{\partial x}(p_1).\]
Since
\[\max_{\Omega\times[0,t_0]}\{u_2-u_1-\eta\psi\}=0=u_2(p_1)-u_1(p_1)-\eta\psi(p_1),\]
we then have
\[\frac{\partial u_1}{\partial x}(p_1)=\frac{\partial u_2}{\partial x}(p_1)-\eta\frac{\partial\psi}{\partial x}(p_1).\]
Consequently,
\[\max\{\mathcal{A}[u_1](p_1),\frac{\partial u_1}{\partial x}(p_1)-1\}\leq\max\{\mathcal{A}[u_2](p_1),\frac{\partial u_2}{\partial x}(p_1)-1\}-\eta\frac{\partial\psi}{\partial x}(p_1).\]
But this is impossible since both $u_1$ and $u_2$ are solutions. This impossibility implies that $u_1\equiv u_2$.
\end{proof}

\section{Numerical Results} \label{sect_numer}
The above results makes the problem solved in theory to a certain extent, but the practical application requires us to give the specific description of the value function. In this section, we present a numerical example of problem (\ref{888}). The values of the parameters are as follows, some of which are selected from \cite{A19}.
\begin{table}[h]
\centering
\caption{Values of parameters}
\begin{tabular}{ccccccccc}
 \hline
 $\mu$ & $\sigma$ & $\nu$ & $m$ & $T$  & r\\
  \hline
  0.05 & 0.2 & 0.5 & 0.3 & 1 & 0.03  \\
  \hline
\end{tabular}
\label{t1}
\end{table}

Figure 1 shows the relationship between the value function $\Phi$ at $t=0$, initial carbon emission over allowance $X_0$ and initial unit price of purchase carbon allowance $S_0$. When $X_0$ is less than 0, which means that there are surplus carbon allowances at this time, the enterprise does not need to take emission reduction actions, and the emission reduction cost is 0. The minimum expectation of cost increases when $X_0$ or $S_0$ increases. This indicates that higher levels of excess emissions result in a more substantial reduction burden. Consequently, the cost of emission reduction for the company escalates. Moreover, a higher unit price for purchasing carbon allowances translates to increased costs in acquiring them. Therefore, the overall expense for companies in reducing emissions also increases. In addition, higher unit price of purchase carbon allowance means higher cost of purchasing carbon allowances. Therefore, the higher the cost of reducing emissions for companies will be. Figure 2 shows the location of the free boundary surface. It can be seen from the figure that for a fixed time $t$, the free boundary increases with the price of carbon allowances. This implies that with a higher price for allowances, the enterprise is less inclined to diminish the excess emissions by purchasing additional quotas.
\begin{figure}[ht]
\centering
\begin{minipage}[b]{0.42\linewidth}
\includegraphics[scale=0.4]{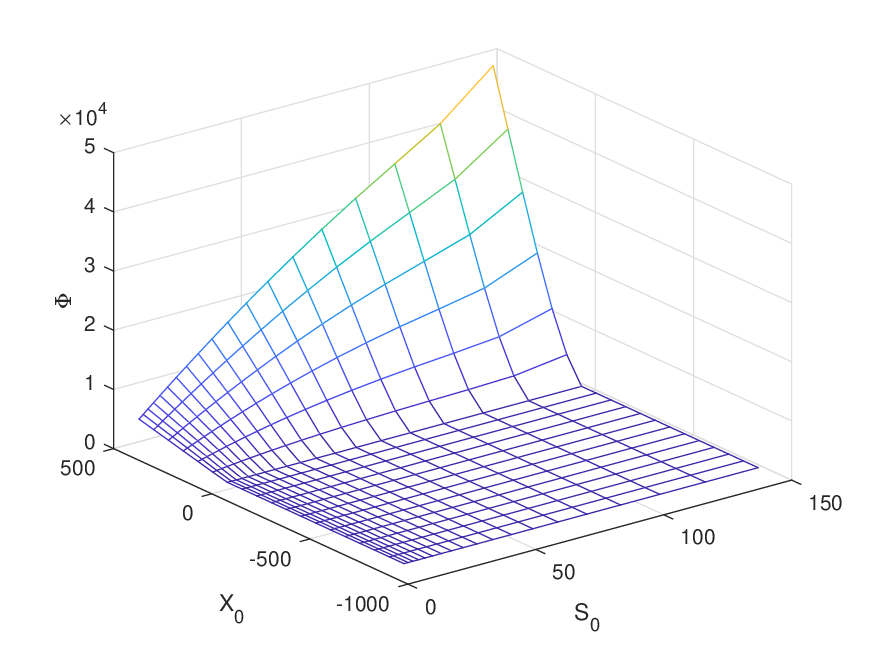}
\caption{Relationship between $\Phi$, $X_0$ and $S_0$}
\label{p1}
\end{minipage}
\quad
\begin{minipage}[b]{0.42\textwidth}
\includegraphics[scale=0.4]{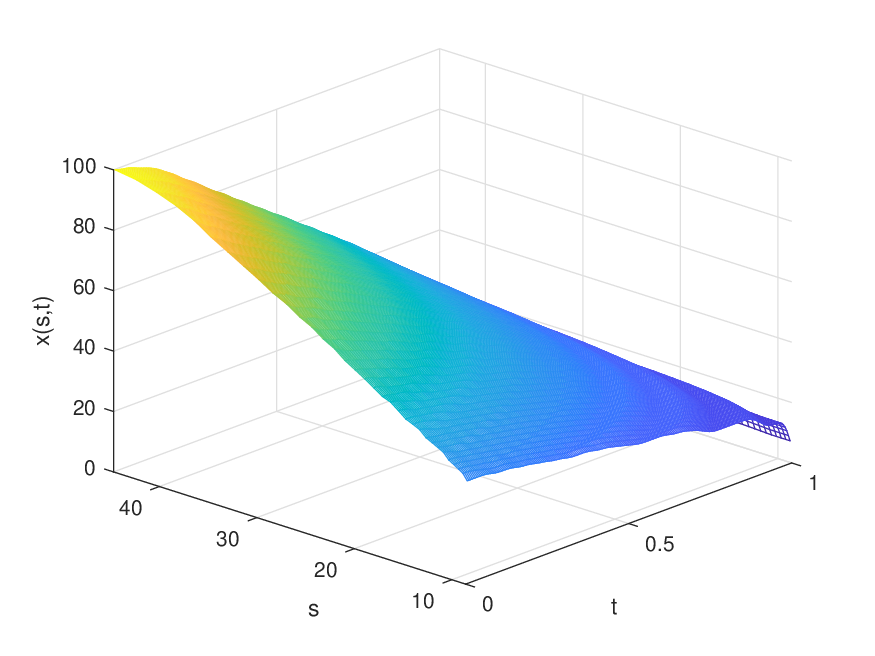}
\caption{Free boundary surface}
\label{p2}
\end{minipage}
\end{figure}

Figure 3 shows the numerical results of (\ref{10}) at $t=T$. In addition, from the previous discussion, we know that the optimal self-control strategy $a_t^*=\frac{1}{2m}\frac{\partial \Phi}{\partial x}(X_t,S_t,t)=\frac{S_{T-t}}{2m}v(X_{T-t},S_{T-t},T-t)$. Therefore, we can calculate the optimal self-control strategy as shown in Figure 4.  From Figure 4, it's evident that when there is a surplus of carbon allowances (i.e., $X_0<0$), the enterprise isn't required to undertake emission reduction measures. As $X_0$ increases, the enterprise's optimal self-regulation becomes more significant, reflecting the increased task of emission reduction. However, once $X_0$ reaches a certain threshold, the optimal control of the enterprise remains unchanged. At this point, the cost of self-regulation exceeds the cost of purchasing carbon allowances, making the latter the optimal choice for the enterprise. Additionally, for a fixed $X_0$, as $S_0$ increases, the enterprise's optimal self-regulation also increases. In such instances, the higher cost associated with purchasing carbon allowances makes enterprises more inclined to adopt self-regulation strategies.

\begin{figure}[ht]
\centering
\begin{minipage}[b]{0.42\linewidth}
\includegraphics[scale=0.4]{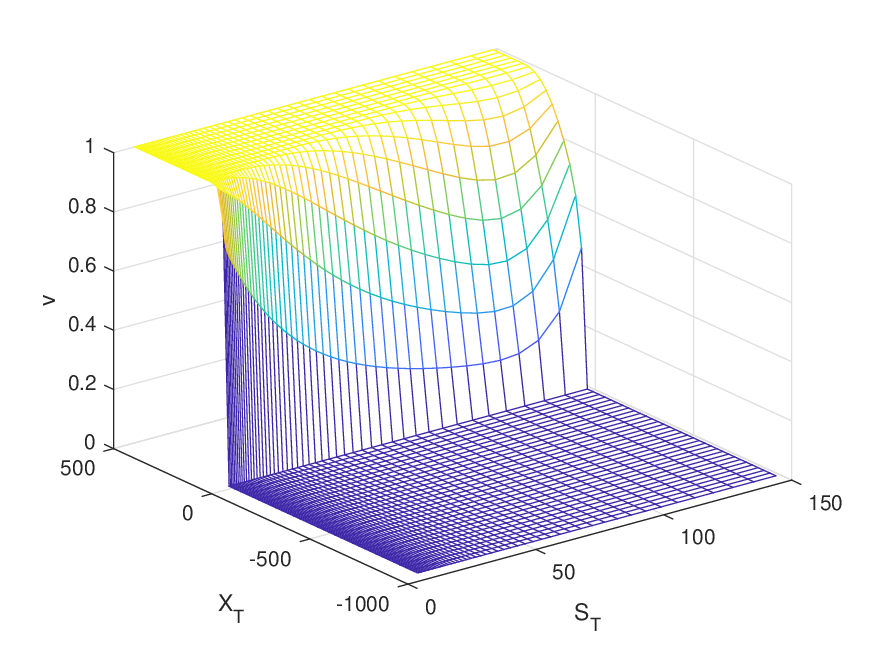}
\caption{Relationship between $v$, $X_0$ and $S_0$}
\label{p3}
\end{minipage}
\quad
\begin{minipage}[b]{0.42\textwidth}
\includegraphics[scale=0.4]{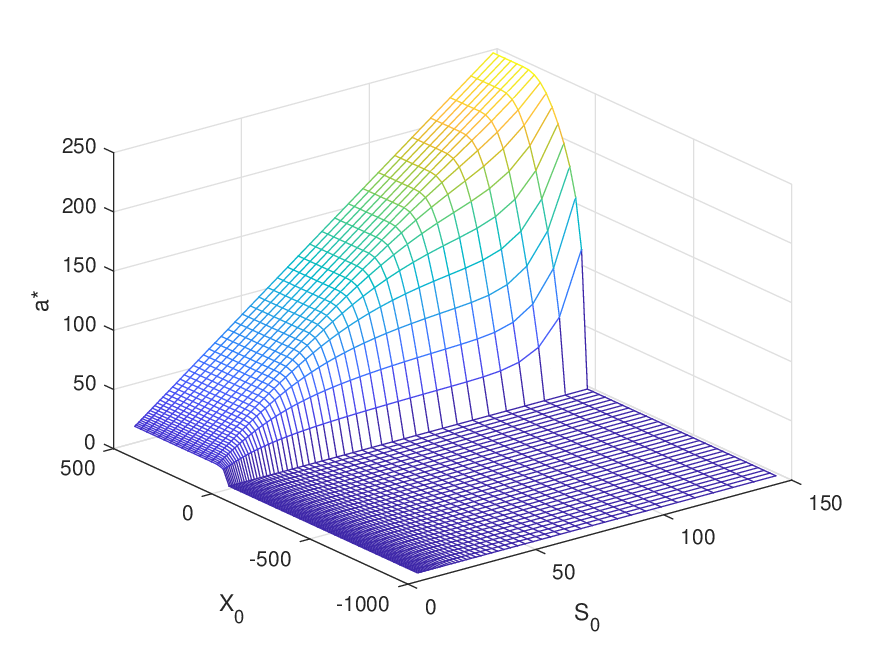}
\caption{Optimal self-control strategy}
\label{p4}
\end{minipage}
\end{figure}
Figures 5-8 show the relationship between the value function and the parameters. A larger $m$ means lower emission reduction efficiency, so the cost of emission reduction will be higher. Therefore, from the perspective of the long-term development of the company, the introduction of advanced emission reduction technology will help reduce the cost of emission reduction. In addition, the increase in the volatility $\nu$ of excess emissions will also increase the emission reduction cost of the enterprise. In terms of the price of carbon allowances, the higher the average growth rate $\mu$, the higher the emission reduction cost, while the higher the volatility $\sigma$, the lower the emission reduction cost.
\begin{figure}[H]
\centering
\begin{minipage}[b]{0.42\linewidth}
\includegraphics[scale=0.4]{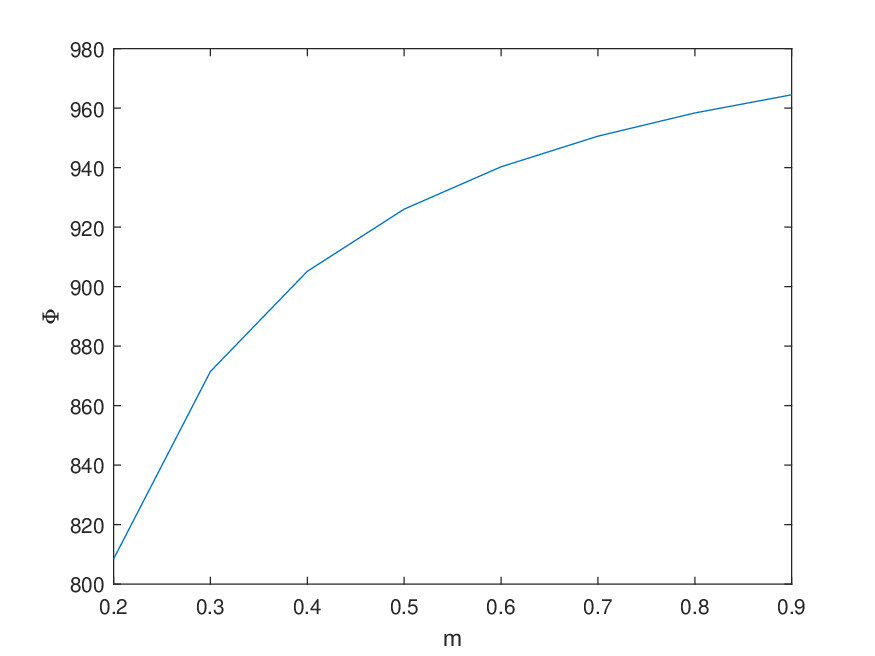}
\caption{Relationship between $\Phi$ and $m$}
\label{p5}
\end{minipage}
\quad
\begin{minipage}[b]{0.42\textwidth}
\includegraphics[scale=0.4]{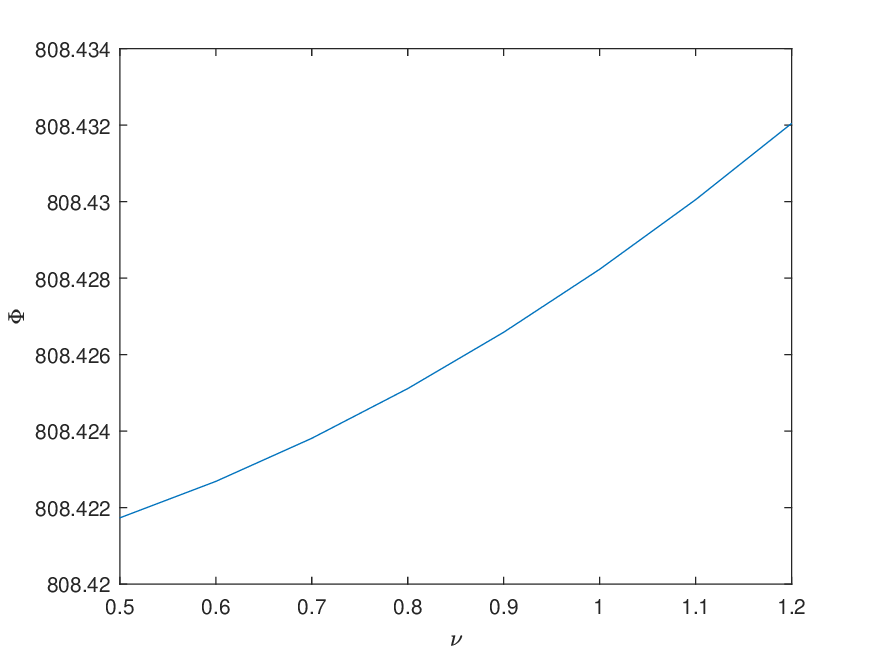}
\caption{Relationship between $\Phi$ and $\nu$}
\label{p6}
\end{minipage}
\end{figure}

\begin{figure}[H]
\centering
\begin{minipage}[b]{0.42\linewidth}
\includegraphics[scale=0.4]{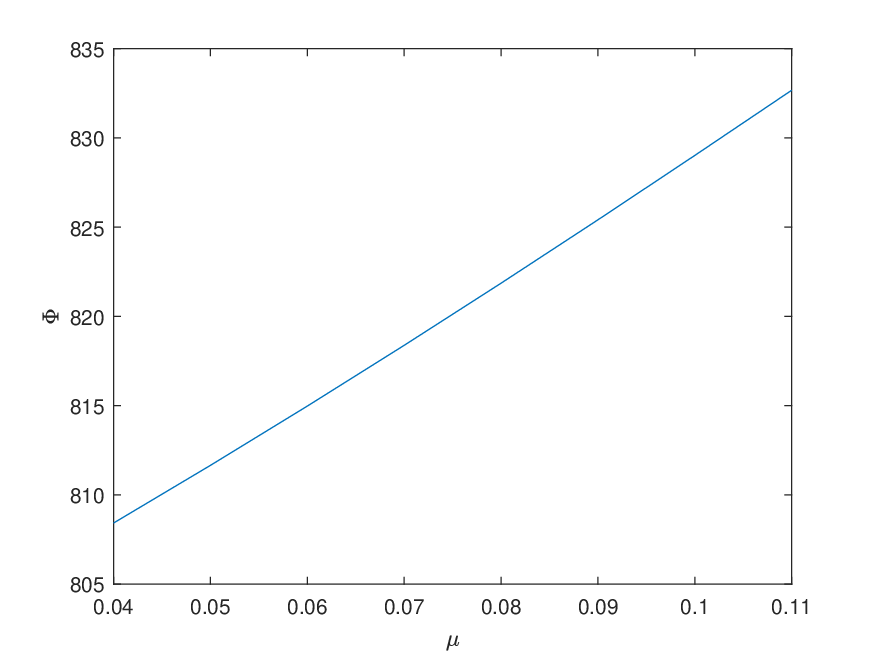}
\caption{Relationship between $\Phi$ and $\mu$}
\label{p7}
\end{minipage}
\quad
\begin{minipage}[b]{0.42\textwidth}
\includegraphics[scale=0.4]{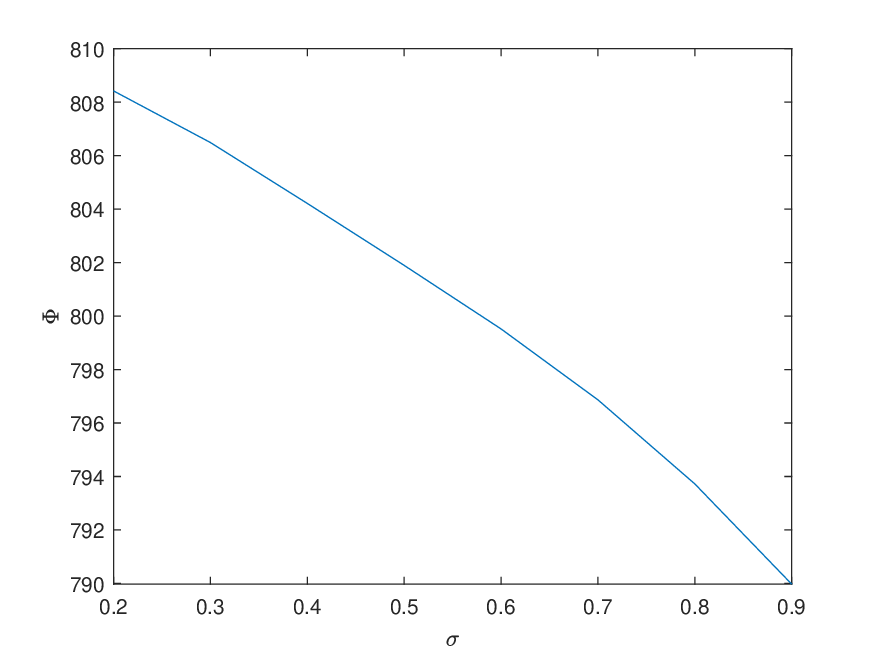}
\caption{Relationship between $\Phi$ and $\sigma$}
\label{p8}
\end{minipage}
\end{figure}

\section{Concluding remarks}
We study an optimal stochastic control problem for carbon emission reduction of enterprises, in which the enterprise can simultaneously reduce its emissions and buy
carbon allowances at any time. Then we establish a model with the goal of minimizing the total cost of emissions reduction for the enterprise. Because the enterprise can purchase the carbon allowances at any time, the problem is transformed from a standard stochastic control problem in the previous study into a singular control problem. We then derived the corresponding HJB equation, which is a parabolic variational inequality on two dimensional state space with gradient barrier, so that the free boundary is a surface. By proving a verification theorem, we show that the solution of the variational inequality is the value function.  By means of the penalty method and the comparison principle, we get the existence and uniqueness of the solution, so the problem is solved theoretically to a certain extent. Finally, we give the visual description of value function and free boundary by numerical results.


\bibliographystyle{siamplain}
\bibliography{references}

\end{document}